\documentclass[11pt]{amsart}
\usepackage[cp1252]{inputenc}
\usepackage{amsmath,amssymb,amsxtra,amsthm}
\usepackage{txfonts}
\usepackage{mathrsfs,bbm,stmaryrd}
\usepackage{graphicx,color}
\usepackage{tikz,tikz-3dplot,pgfplots}
\usetikzlibrary{calc,3d}
\usepackage{citeref}
\usepackage[active]{srcltx}
\usepackage{hyperref}
\hypersetup{hidelinks}

\tdplotsetmaincoords{70}{110}

\def\XXint#1#2#3{{\setbox0=\hbox{$#1{#2#3}{\int}$ }
		\vcenter{\hbox{$#2#3$ }}\kern-.6\wd0}}

\DeclareMathOperator{\Capac}{Cap}
\DeclareMathOperator{\divop}{div}

\newtheorem{thm}{Theorem}[section]

\newtheorem{lem}[thm]{Lemma}

\newtheorem{rem}[thm]{Remark}

\newcommand{\ct}[1]{\langle {#1}\bigangle \lower.3ex\hbox{$_{t}$}}
\newcommand{\lt}[1]{[ {#1}]
	\lower.3ex\hbox{$_{t}$}}

\begin{document}
	
	\title[Sharp $p$-Capacity Estimates via Quermassintegrals in Hyperbolic Space]{Sharp $p$-Capacity Estimates via Quermassintegrals in Hyperbolic Space}

	\author{Xiaoshang Jin}
	\address{Xiaoshang Jin: School of Mathematics and Statistics, Huazhong University of Science and Technology, Wuhan, Hubei 430074, China}
	\email{jinxs@hust.edu.cn}
	
	\author{Yao Wan}
	\address{Yao Wan: Department of Mathematics, The Chinese University of Hong Kong, Shatin, Hong Kong}
	\email{yaowan@cuhk.edu.hk}

	\author{Jie Xiao}
	\address{Jie Xiao: Department of Mathematics and Statistics,
		Memorial University of Newfoundland, St. John's, NL A1C 5S7, Canada}
	\email{jxiao@math.mun.ca}
	
	\thanks{X. Jin was supported by `the Fundamental Research Funds for the Central Universities', HUST: \# 2025BRSXB002; Y. Wan was supported by Hong Kong RGC grant (Early Career Scheme) of Hong Kong \# 24304222 and \# 14300623, and NNSF of China \# 12222122; J. Xiao was supported by NSERC of Canada \# 202979.}


	\begin{abstract}
		This paper establishes sharp upper bounds for $p$-capacities $\mathrm{Cap}_{1<p<\infty}$ in the hyperbolic space $\mathbb{H}^n$ through hyperbolic quermassintegrals and effective curvature radii. The quermassintegral comparisons involve $W_{n-1}$, $W_{k+1}+k(n+1-k)^{-1}W_{k-1}$, and the pair $W_1\mid W_2$. For star-shaped, mean-convex or h-convex hypersurfaces, inverse mean curvature flow further produces curvature radii determined by $L^q$-averages of the normalized mean curvature and by moments of its squared hyperbolic excess. These radii convert the resulting estimates into sharp geodesic-ball comparisons for the capacity-to-area ratio. In the range $p>2m+1$, an interpolating radius combines the $2m$-th curvature-excess radius with the $L^\infty$ curvature scale, thereby linking the finite-moment and supremum regimes. Equality in the sharp comparisons characterizes geodesic balls.
	\end{abstract}
	
	\keywords{Hyperbolic spaces, capacities, quermassintegrals, normalized mean curvatures, measures, effective curvature radii}

	\subjclass[2010]{53C21; 31C15; 53C44; 58J35}

	\maketitle
	{\tableofcontents}
	
	\arraycolsep=1pt
	\numberwithin{equation}{section}
	
	\thanks{}

	\arraycolsep=1pt
	
	\numberwithin{equation}{section}
	
\section{Introduction and Background}

We begin by establishing the geometric and notation conventions utilized throughout this paper. Let $\mathbb{H}^n = (0,\infty) \times \mathbb{S}^{n-1}$ represent the $n$-dimensional hyperbolic space ($n \ge 2$) equipped with the standard Riemannian metric:
\[
\mathsf{g} = d\rho^2 + (\sinh^2\rho)\,d\theta^2
\]
expressed in polar coordinates $(\rho,\theta)$. Under this framework, we define:

\begin{itemize}
	\item $\nu$ and $\mu$: The $n$-dimensional Riemannian volume and $(n-1)$-dimensional Riemannian-Hausdorff measure induced by $\mathsf{g}$, respectively.
	\item $\nabla$ and $\Delta = \nabla \cdot \nabla$: The gradient and Laplacian operators associated with $\mathsf{g}$.
	\item $\divop \mathbf{F} = \nabla \cdot \mathbf{F}$: The divergence of a vector field $\mathbf{F}$ with respect to $\mathsf{g}$.
	\item $C_0^\infty(\mathbb{H}^n)$: The space of smooth functions with compact support in $\mathbb{H}^n$.
\end{itemize}

\begin{figure}[htbp]
	\centering
	\begin{tikzpicture}[scale=1.8]
		\draw[thick] (0,0) circle (1);
		\foreach \a in {0,30,...,150} {
			\draw[blue, opacity=0.6] (0,0) -- ({cos(\a)},{sin(\a)});
			\draw[blue, opacity=0.6] (0,0) -- ({-cos(\a)},{-sin(\a)});
		}
		\foreach \r in {0.2,0.4,0.6,0.8} {
			\draw[blue, opacity=0.6] (0,0) circle (\r);
		}
		\node at (0,-1.2) {\small Poincar\'e disk model for $(\mathbb{H}^2, \mathsf{g})$};
	\end{tikzpicture}
\end{figure}

Next, we view a compact subset $K \subset \mathbb{H}^n$ as a physical capacitor embedded within a space of constant negative curvature. For $p \in [1,\infty)$, the variational $p$-capacity (representing variational minimal energy or heat-loss profiles) of $K$ is defined by:
\[
\Capac_{p}(K) = \inf\left\{ \int_{\mathbb{H}^n} |\nabla f|^p \, d\nu : f \in C_0^\infty(\mathbb{H}^n), \; f = 1 \text{ on } K \right\}.
\]
The geometric endpoint case where $p=1$ corresponds directly to a structural perimeter minimization problem:
\[
\Capac_1(K) = \inf \bigl\{ \mu(\partial \Omega) : K \subset \Omega, \, \text{$\Omega$ is a bounded open set with smooth boundary} \bigr\}.
\]
In particular, when $K = \bar{B}(r)$ is a closed geodesic ball of radius $r \in [0,\infty)$ centered at the origin, its exact $p$-capacity is explicitly computable (cf. \cite{Gr, J, JX}) and given by:
\begin{equation}\label{1.1}
	\Capac_p\big(\bar{B}(r)\big) = \begin{cases}
		\left(\int_r^\infty \big(\omega_{n-1}\sinh^{n-1}t\big)^{\frac{1}{1-p}}\,dt\right)^{1-p} & \text{for } p \in (1,\infty); \\
		\omega_{n-1}\sinh^{n-1}r & \text{for } p = 1,
	\end{cases}
\end{equation}
where $\omega_{n-1} = \frac{2\pi^{n/2}}{\Gamma(n/2)}$ represents the surface area of the standard Euclidean unit sphere $\mathbb{S}^{n-1}$:
\[
\omega_{n-1} = \begin{cases} 
	\frac{2\pi^{n/2}}{(\frac{n}{2}-1)!} & \text{for even } n; \\
	\frac{2(2\pi)^{(n-1)/2}}{(n-2)!!} & \text{for odd } n.
\end{cases}
\]

For a smooth compact domain $K \subset \mathbb{H}^n$, the structural invariants known as hyperbolic quermassintegrals $W_k(K)$ are defined recursively for $k \in \{1,\dots,n-1\}$ by:
\[
\begin{cases}
	W_0(K) = \nu(K) = |K|; \\
	W_1(K) = n^{-1}\mu(K) = n^{-1} |\partial K|; \\
	W_{k+1}(K) = n^{-1}\int_{\partial K} \mathrm{p}_k \, d\mu - k(n+1-k)^{-1}W_{k-1}(K),
\end{cases}
\]
where $$\mathrm{p}_k = \sum_{i_1<\dots<i_k}\binom{n-1}{k}^{-1}\kappa_{i_1}\dots\kappa_{i_k}$$ denotes the normalized $k$-th mean curvature of the boundary $\partial K$. 

If $K$ is strictly convex in the background Euclidean metric $d\rho^2 + \rho^2\,d\theta^2$, its parallel outer domain at a distance $t \ge 0$ is defined as 
$$K_t = \{x \in \mathbb{H}^n : d(x,K) \le t\}.
$$ This parallel profile satisfies the hyperbolic Steiner-type formula (cf. \cite{Ko91}):
\begin{equation}\label{1.2}
	\begin{cases}
		W_{n}(K) = n^{-1}{\omega_{n-1}}; \\
		\int_{\partial K} \mathrm{p}_{n-1}\,d\mu = \omega_{n-1} + \frac{n(n-1)}{2}W_{n-2}(K); \\
		n W_1(K_t) = |\partial K_t| = \sum_{k=0}^{n-1} \left(\int_{\partial K}\binom{n-1}{k} \mathrm{p}_k \, d\mu\right)(\cosh^{n-1-k} t) (\sinh^k t).
	\end{cases}
\end{equation}

Crucially, as shown in \cite[Theorem 7]{LLX} and \cite[(5.5)]{JX}, when $\partial K$ is a smooth convex boundary, the variational $p$-capacity satisfies a geometric inequality bound known as the capacity-to-quermassintegral principle:
\begin{equation}\label{1.3}
	\Capac_{p}(K) \le \begin{cases} 
		\left( \int_0^\infty \big(nW_1(K_t)\big)^{\frac{1}{1-p}} \, dt \right)^{1-p} & \text{for } p \in (1,\infty); \\
		\inf_{t \in (0,\infty)} \big(nW_1(K_t)\big) & \text{for } p = 1.
	\end{cases}
\end{equation}
This bound acts as a core bridge linking potential theory to foundational settings in general relativity and statistical mechanics. Motivated by this principle, the goal of this paper is to leverage deeper combinations of the $W_k$-functionals to establish optimal, sharp upper bounds on $\Capac_p$. The remaining sections are structured as follows:

\begin{itemize}
	\item \textbf{Section 2:} We establish a sharp order relation to dominate $\Capac_{1<p<\infty}$ in terms of the specific single integral invariant $W_{n-1}$ (Theorem~\ref{thm2.1}).
	\item \textbf{Section 3:} We exploit the potential profiles of geodesic balls to cleanly control $\Capac_{1<p<\infty}$ via the recurrent grouping $W_{k+1} + k(n+1-k)^{-1}W_{k-1}$ (Theorem~\ref{thm2.3}).
	\item \textbf{Section 4:} By tracking structural geometric evolution profiles along the inverse mean curvature flow (IMCF), we optimally bound the sub-quadratic capacity range $\Capac_{1<p\le 2<n}$ via the ratio combination $W_1 \mid W_2$ (Theorem~\ref{thm3.1}).
	\item \textbf{Section 5:} We combine IMCF monotonicity with effective curvature radii to obtain sharp geodesic-ball comparisons for the capacity-to-area ratio, together with complementary bounds involving $W_1$, $W_2$, and $W_0$ (Theorem~\ref{thm3.3}).
\end{itemize}

\section{Dominating ${\rm Cap}_{1<p<\infty}$ through $W_{n-1}$}\label{s2}

In order to compare two different geometric quantities, we introduce the following order relation. Given two geometric quantities $Q_1(K)$ \& $Q_2(K)$ defined on any compact subset $K$ of $\mathbb{H}^n$, we write
\[
Q_1(K) \underset{R}{\lesssim} Q_2(K)
\]
provided that the geodesic ball of the same $Q_1$-value as $K$ has radius not larger than the geodesic ball of the same $Q_2$-value as $K$. In other words,
$$
Q_1(K) \underset{R}{\lesssim} Q_2(K) \Leftrightarrow h_1^{-1} ( Q_1(K) )\leq h_2^{-1} ( Q_2(K) )\ \ \text{where}\ \
h_i(r) = Q_i\big(\bar{B}(r)\big)\ \ \text{for}\ \ i=1,2.
$$

The full quermassintegral inequalities in $\mathbb{H}^n$ state that (cf. Wang-Xia~\cite{WX14} \&  Hu-Li-Wei~\cite{HLW22})
\begin{equation*}
W_k(K) \geq f_k \circ f_l^{-1}\bigl(W_l(K)\bigr)\ \ \text{for}\ \  \begin{cases} 0 \leq l < k \leq n-1;\\
\text{smooth bounded $h$-convex domains $K\subset\mathbb H^n$};\\
f_k(r)=W_k\big(\bar{B}(r)\big).
\end{cases}
\end{equation*}
This immediately implies the chain of radius comparisons
\begin{equation}
\label{2.1}
W_0(K)=|K|\underset{R}{\lesssim} W_1(K)=n^{-1}|\partial K|\underset{R}{\lesssim} W_2(K)\underset{R}{\lesssim}\cdots\underset{R}{\lesssim}W_{n-2}(K) \underset{R}{\lesssim}W_{n-1}(K).
\end{equation}
It is natural for us to ask where the hyperbolic $p$-capacity fits into this sequence. Interestingly, from \cite{JX} it follows that for any smooth compact domain $K \subset \mathbb{H}^n$ there holds
\begin{equation}\label{isocap}
W_0(K) \underset{R}{\lesssim} {\rm Cap}_{p}(K).
\end{equation}
Even more interestingly, we find the following $(1,\infty)\ni p$-result which, along with not only \eqref{isocap} but also \eqref{1.3} for $p=1$, not only enriches \eqref{2.1} at large but also exists as the first principle for capacities-to-quermassintegrals.

\begin{thm}\label{thm2.1}
	Given $p\in (1,\infty)$, let $K \subset \mathbb{H}^n$ be a compact domain with smooth strictly convex boundary $\partial K$. Then
	\begin{equation}\label{cap-pn}
	{\rm Cap}_{p}(K) \underset{R}{\lesssim} W_{n-1}(K),
	\end{equation}
	with equality iff $K$ is a geodesic ball.
\end{thm}

\begin{proof} \eqref{cap-pn}'s equality follows immediately from  that in the quermassintegral inequalities. So, it is enough to verify \eqref{cap-pn}'s inequality according to the dimension situations: $n=2$; $n\ge 3$.
	
	\begin{itemize}
	\item If $n=2$, then the Gauss-Bonnet theorem together with the isoperimetric inequality yields
	\begin{align*}
	\int_{\partial K}{\rm p}_1 d\mu=|K|+2\pi\leq \sqrt{|\partial K|^2+4\pi^2}.
	\end{align*}
	Upon using \eqref{1.2}-\ref{1.3}, we obtain
	\begin{align*}
	{\rm Cap}_{p}(K)
	&\leq \left( \int_0^\infty (|\partial K|\cosh t+ \sqrt{|\partial K|^2+4\pi^2}\sinh t)^{\frac{1}{1-p}} \, dt \right)^{1-p}\\
	& =  \left( \int_{\text{arsinh}\frac{|\partial K|}{2\pi}}^\infty (2\pi\sinh t)^{\frac{1}{1-p}} \, dt \right)^{1-p}\\
	& \equiv F_p(|\partial K|).
	\end{align*}
	The function $F_p(s)$ not only increases in $s$ but also satisfies
	\begin{align*}
	F_p(|\partial {\bar B}(r)|)={\rm Cap}_{p}\big({\bar B}(r)\big).
	\end{align*}
	Letting
	$$
	\begin{cases} g_p(r) \equiv {\rm Cap}_{p}\big({\bar B}(r)\big);\\
	h(r) \equiv|\partial {\bar B}(r)|,
	\end{cases}
	$$
	gives
	$$
	\begin{cases} F_p(s) = g_p \circ h^{-1}(s);\\
	
	{\rm Cap}_{p}(K)\leq g_p\circ h^{-1}(|\partial K|)
	\Leftrightarrow\, {\rm Cap}_{p}(K) \underset{R}{\lesssim} |\partial K|=2W_1(K).
	\end{cases}
	$$
	
	\item If $n\geq 3$, then by \cite[Theorem 1.3]{HL23}\footnote{In the original notes, it was mistakenly claimed that the inequality also holds for $k=n-2$.}, we have
	\begin{equation*}
	\int_{\partial K} {\rm p}_k \, d\mu
	\leq \Biggl(\int_{\partial K} {\rm p}_{n-1} \, d\mu\Biggr)^{\frac{k}{n-1}}
	\Biggl(\Biggl(\int_{\partial K} {\rm p}_{n-1} \, d\mu\Biggr)^{\frac{2}{n-1}} - \omega_{n-1}^{\frac{2}{n-1}}\Biggr)^{\frac{n-1-k}{2}}\ \forall\ k\in\{0,1,...,n-3\},
	\end{equation*}
	with equality iff $K$ is a geodesic ball. Since $\partial K$ is strictly convex, the quermassintegral inequalities in \cite{AHL20, BGL19} give
	\begin{align*}
	W_{n-1}(K) \geq f_{n-1} \circ f_{l}^{-1}\bigl(W_{l}(K)\bigr)\ \forall\  l\in\{0,1,...,n-2\}.
	\end{align*}
	From the definition of quermassintegrals, we get
	$$
	\begin{cases}
	\int_{\partial K} {\rm p}_{n-1}\,d\mu=n W_n(K)+2^{-1}{n(n-1)}W_{n-2}(K)
	=\omega_{n-1}+2^{-1}{n(n-1)}W_{n-2}(K);\\
	\int_{\partial K} {\rm p}_{n-2}\,d\mu=n W_{n-1}(K)+3^{-1}{n(n-2)}W_{n-3}(K).
	\end{cases}
	$$
	In conclusion, we obtain
	\begin{align*}
	\int_{\partial K} {\rm p}_k\, d\mu \underset{R}{\lesssim} W_{n-1}(K)\ \forall\  k\in\{0,1,...,n-1\}.
	\end{align*}
	More precisely, let $\xi(s)$ be defined by
	$$\xi(W_{n-1}\big({\bar B}(r)\big))=\int_{\partial {\bar B}(r)}{\rm p}_{n-1}\,d\mu.
	$$ Then
	\begin{align}\label{pk-W}
	\int_{\partial K} {\rm p}_k \, d\mu
	\leq \xi\big(W_{n-1}(K)\big)^{\frac{k}{n-1}}
	\left(\xi\big(W_{n-1}(K)\big)^{\frac{2}{n-1}} - \omega_{n-1}^{\frac{2}{n-1}}\right)^{\frac{n-1-k}{2}}
	\ \forall\  k\in\{0,1,...,n-1\}.
	\end{align}
	Substituting this into the Steiner formula yields
	\begin{align*}
	|\partial K_t|
	&\leq \frac{\cosh^{n-1}t  \left(\big(\xi\big(W_{n-1}(K)\big)^{\frac{2}{n-1}}-\omega_{n-1}^{\frac{2}{n-1}}\right)^{\frac{n-1}{2}}}{\left(\sum\limits_{k=0}^{n-1}\binom{n-1}{k}
	\left(\sqrt{\frac{\xi\big(W_{n-1}(K)\big)^{\frac{2}{n-1}} }{\xi\big(W_{n-1}(K)\big)^{\frac{2}{n-1}} -\omega_{n-1}^{\frac{2}{n-1}}}}\tanh t\right)^{k}\right)^{-1}}\\
	&=\frac{ \cosh^{n-1}t  \left(\xi\big(W_{n-1}(K)\big)^{\frac{2}{n-1}}-\omega_{n-1}^{\frac{2}{n-1}}\right)^{\frac{n-1}{2}}}{
	\left(1+\sqrt{\frac{\xi\big(W_{n-1}(K)\big)^{\frac{2}{n-1}} }{\xi\big(W_{n-1}(K)\big)^{\frac{2}{n-1}} -\omega_{n-1}^{\frac{2}{n-1}}}}\tanh t\right)^{1-n}}\\
	&=\left(\sqrt{\xi\big(W_{n-1}(K)\big)^{\frac{2}{n-1}} -\omega_{n-1}^{\frac{2}{n-1}}}\cosh t+\sqrt{\xi\big(W_{n-1}(K)\big)^{\frac{2}{n-1}} }\sinh t\right)^{n-1}.
	\end{align*}
	Thus,
	\begin{align*}
	&{\rm Cap}_{p}(K)\\
	&\ \leq\left(\int_0^{\infty}
	\left(\sqrt{\xi\big(W_{n-1}(K)\big)^{\frac{2}{n-1}} -\omega_{n-1}^{\frac{2}{n-1}}}\cosh t+\sqrt{\xi\big(W_{n-1}(K)\big)^{\frac{2}{n-1}} }\sinh t\right)^{\frac{n-1}{1-p}} dt\right)^{1-p}\\
	&\ \equiv G_p\big(W_{n-1}(K)\big) .
	\end{align*}
Note that the function $G_p(s)$ not only increases in $s$ but also satisfies
	\begin{align*}
	G_p\left(W_{n-1}\big({\bar B}(r)\big)\right)={\rm Cap}_{p}\big({\bar B}(r)\big).
	\end{align*}
	Thus, we achieve
	$$
	\begin{cases}
	G_p=g_p\circ f_{n-1}^{-1};\\
	{\rm Cap}_{p}(K)\leq g_p\circ f_{n-1}^{-1}\big(W_{n-1}(K)\big)
	\Leftrightarrow\, {\rm Cap}_{p}(K) \underset{R}{\lesssim} W_{n-1}(K).
	\end{cases}
	$$
\end{itemize}

\end{proof}

\begin{rem}\label{rem2.2} Two comments are in order.
	\begin{itemize}
		\item When $\partial K$ is not only smooth but also horospherically convex, from \cite[Theorem 5.1]{JX} it follows that
	\begin{align}\label{cap-rV}
	{\rm Cap}_{p}(K) \underset{R}{\lesssim} {\rm RV}(K)=\lim\limits_{t\to\infty} \frac{|K_t|}{|B(t)|}=\sum_{i=0}^{n-1}\binom{n-1}{i} \frac{\int_{\partial K}{\rm p}_i\, d\mu}{\omega_{n-1}}.
	\end{align}
Meanwhile, (\ref{pk-W}) further implies
	\begin{align}\label{2.6}
	{\rm RV}(K)\underset{R}{\lesssim} W_{n-1}(K).
	\end{align}
Interestingly, a combination of \eqref{cap-rV}-\eqref{2.6} can refine \eqref{cap-pn} into
$$
{\rm Cap}_{p}(K) \underset{R}{\lesssim} {\rm RV}(K)\underset{R}{\lesssim} W_{n-1}(K).
$$

\item Even more interestingly, we find
\begin{equation}
\label{2.7}
\min\left\{{\rm RV}(K):\ \sum_{i=0}^{n-1}\left(\frac{\int_{\partial K}{\rm p}_i\, d\mu}{\omega_{n-1}}\right)^{-1}=1\right\}=\left(\sum_{i=0}^{n-1}\binom{n-1}{i}^\frac12\right)^2.
\end{equation}
Indeed, via not only letting
$$
\begin{cases}\lambda_i^{-1}=\frac{\int_{\partial K}{\rm p}_i\, d\mu}{\omega_{n-1}};\\
\sum_{i=0}^{n-1}\lambda_i=1\ \ {\color{blue}\text{whose case $n=3$ visualized as}}\ \
\begin{tikzpicture}[tdplot_main_coords, scale=3]
\draw[->] (0,0,0) -- (1.2,0,0) node[anchor=north] {$$};
\draw[->] (0,0,0) -- (0,1.2,0) node[anchor=west] {$$};
\draw[->] (0,0,0) -- (0,0,1.2) node[anchor=south] {$$};
\filldraw[fill=blue!20, draw=blue, opacity=0.7]
(1,0,0) -- (0,1,0) -- (0,0,1) -- cycle;
\node at (1,0,0) [below right] {$(1,0,0)$};
\node at (0,1,0) [left] {$(0,1,0)$};
\node at (0,0,1) [above] {$(0,0,1)$};
\end{tikzpicture}
\end{cases}
$$
but also applying the Lagrange multiplier method to
$$
{\rm RV}(K)=\sum_{i=0}^{n-1}\binom{n-1}{i}\lambda_i^{-1},
$$
we see that the minimizing problem of \eqref{2.7} can be resolved by differentiating
$$
\Phi=\sum_{i=0}^{n-1}\binom{n-1}{i}\lambda_i^{-1}+t\Bigg(\sum_{i=0}^{n-1}\lambda_i-1\Bigg).
$$
with respect to $\lambda_i$. For this, we get
	$$
	\begin{cases}
	\frac{\partial \Phi}{\partial \lambda_i}=-{\binom{n-1}{i}}\lambda_i^{-2}+t=0\ \ \forall\ \ i\in\{0,1,...,n-1\};\\
	\frac{\partial\Phi}{\partial t}=\sum_{i=0}^{n-1} \lambda_i-1=0,
	\end{cases}
	$$
	thereby having not only
	$$
	\begin{cases}
		t=\left(\sum_{i=0}^{n-1}{\binom{n-1}{i}}^\frac{1}{2}\right)^{2};\\
	\lambda_i=\left({t}^{-1}{\binom{n-1}{i}}\right)^\frac{1}{2}=\left(\sum_{j=0}^{n-1}{\binom{n-1}{j}}^\frac{1}{2}\right)^{-1}{\binom{n-1}{i}}^\frac{1}{2}\ \ \forall\ \ i\in\{0,1,...,n-1\},
	\end{cases}
	$$
	but also the minimum as desired in \eqref{2.7}.

\end{itemize}
\end{rem}

\section{Dominating ${\rm Cap}_{1<p<\infty}$ through $ W_{k+1}+k(n+1-k)^{-1}W_{k-1}$}\label{s3}

Thanks to
$$
n^{-1}\int_{\partial K} {\rm p}_k \, d\mu=W_{k+1}(K)+k(n+1-k)^{-1}W_{k-1}(K),
$$
plus \eqref{1.3} for $p=1$, the second sharp principle for capacities-to-quermassintegrals can be described below.

\begin{thm}\label{thm2.3}
Given $p\in (1,\infty)$ \&\ $r\in (0,\infty)$, let $K$ be a compact convex body in $\mathbb{H}^{n}$ with smooth boundary $\partial K$. Then there holds
	\begin{equation}\label{2.8}
	{\rm Cap}_p(K)  \leq\frac{\sum\limits_{k=0}^{n-1} \binom{n-1}{k}\int_{\partial K} {\rm p}_k \, d\mu\cdot\int_0^\infty \sinh^{\frac{p(1-n)}{p-1}}(r+\tau)\cdot\cosh^{n-1-k} \tau\cdot  \sinh^k \tau\,d\tau}{\Big(\int_r^\infty \sinh^{\frac{1-n}{p-1}} s\,ds\Big)^p},
	\end{equation}
	with the equality being valid for $K=\bar{B}(r)$.
\end{thm}

\begin{proof} Without loss of generality, we assume that not only $r\in (0,\infty)$ but also $v(x)$ be the $p-$capacity potential of $\bar{B}(r)$. Then
	$$
	\begin{cases} v(x)=\mathrm{v}(|x|)=\frac{I(|x|)}{I(r)};\\ I(\rho)=\int_\rho^\infty \sinh^{\frac{1-n}{p-1}} s\,ds.
	\end{cases}
	$$
	Upon defining
	$$
	\begin{cases}
	g(t)=|\nabla v|_{v=t};\\
	 F(t)=\int_t^1\frac{ds}{g(s)},
	 \end{cases}
	 $$
	we directly calculate
	$$
	\begin{cases}
	g(t)=-\mathrm v'(\mathrm{v}^{-1}(t));\\
	 F(t)=\int_{\mathrm{v}^{-1}(t)}^{\mathrm{v}^{-1}(1)}\frac{\mathrm{v}'(\tau)}{-\mathrm{v}'(\tau)}\,d\tau=\mathrm{v}^{-1}(t)-r,
	 \end{cases}
	$$
	thereby reading off that $F(t)$ is a decreasing function with
	$$
	\begin{cases}
F(1)=0;\\
F(0)=\infty.
\end{cases}
$$
	
	Given a smooth compact domain $K\subset\mathbb{H}^n,$ let
	$$
	\begin{cases}
	u(x)=F^{-1}\big(d_K(x)\big);\\
	d_K(\cdot)=d(\cdot, K)=\text{the distance function of $K$};\\
	 F^{-1}=\text{the inverse of $F$}.
	 \end{cases}
	 $$
	  By truncation at infinity and standard smoothing/density, $u$ may be used as an admissible competitor in the variational characterization of capacity. Then
	$$
	\begin{cases}
	u|_K=1;\\
	u|_\infty =0;\\
	|\nabla u|_{u=t}=g(t).
	\end{cases}
	$$
	As a consequence, we have
	{
	\begin{align*}
&	{\rm Cap}_p(K)\\
&\ \ \leq \int_{K^c} |\nabla u|^p\, d\nu
	\\
	&\ \ =\int_0^1dt\int_{\{u=t\}}|\nabla u|^{p-1}\,d\mu
	\\
	&\ \ =\int_0^1 \big(g(t)\big)^{p-1}\sum_{k=0}^{n-1} \int_{\partial K}\binom{n-1}{k} {\rm p}_k \, d\mu\cdot\big(\cosh^{n-1-k} F(t)\big)\cdot \big(\sinh^k F(t)\big) \,dt
	\\
	&\ \ =\sum_{k=0}^{n-1} \binom{n-1}{k}\int_{\partial K} {\rm p}_k \, d\mu\cdot |\mathrm v'(\mathrm{v}^{-1}(t))|^{p-1}\cdot \Bigg(\frac{\cosh^{n-1-k} (\mathrm{v}^{-1}(t)-r)}{\big(\sinh^k (\mathrm{v}^{-1}(t)-r)\big)^{-1}}\Bigg) \,dt
	\\
	&\ \ =\sum_{k=0}^{n-1} \binom{n-1}{k}\int_{\partial K} {\rm p}_k \, d\mu\cdot\int_{\mathrm{v}^{-1}(0)-r}^{\mathrm{v}^{-1}(1)-r}\frac{\cosh^{n-1-k} \tau\cdot  \sinh^k \tau\cdot v'(r+\tau)}{|\mathrm v'(r+\tau)|^{1-p}} \,d\tau
	\\
	&\ \ =\sum_{k=0}^{n-1} \binom{n-1}{k}\int_{\partial K} {\rm p}_k \, d\mu\cdot\int_0^\infty |\mathrm v'(r+\tau)|^{p}\cosh^{n-1-k} \tau\cdot  \sinh^k \tau\,d\tau
	\\
	&\ \ =\big(I(r)\big)^{-p}\sum_{k=0}^{n-1} \binom{n-1}{k}\int_{\partial K} {\rm p}_k \, d\mu\cdot\int_0^\infty \sinh^{\frac{p(1-n)}{p-1}}(r+\tau)\cdot\cosh^{n-1-k} \tau\cdot  \sinh^k \tau\,d\tau,
	\end{align*}
	}
as desired in \eqref{2.8}.

Of course, \eqref{2.8}'s equality is attainable whenever $K=\bar{B}(r)$.
\end{proof}

\begin{rem} Two comments are in order.
	
	\begin{itemize}
		\item If $\mathsf{r}_{n-1}>0$ is the solution of
		$$W_{n-1}(K)=W_{n-1}\big({\bar B}(\mathsf{r}_{n-1})\big),
		$$
		then applying
		$$ \int_{\partial K} {\rm p}_k\, d\mu \underset{R}{\lesssim} W_{n-1}(K),
		$$
		{equivalently}, 
		$$\int_{\partial K} {\rm p}_k\, d\mu \leq \int_{\partial \bar{B}(\mathsf{r}_{n-1})} {\rm p}_k\, d\mu\ \ \forall\ \  k\in\{0,1,...,n-1\}
		$$
		to Theorem \ref{thm2.3}'s case $r=\mathsf{r}_{n-1}$ yields
		$${\rm Cap}_{p}(K) \leq {\rm Cap}_{p}\big({\bar B}(\mathsf{r}_{n-1})\big).$$
Meanwhile, when $\partial K$ is not only smooth but also horospherically convex, we can utilize \cite[(5.3)]{JX} to get that if $$\mathsf{r}_\infty=(n-1)^{-1}\ln {\rm RV}(K),
	$$
	then
        $$\int_0^\infty \sinh^{\frac{p(1-n)}{p-1}}(\mathsf{r}_\infty+\tau)\cdot |\partial K_{\tau}| \,d\tau
        \leq\omega_{n-1} \int_0^\infty \sinh^{\frac{1-n}{p-1}}(\mathsf{r}_\infty+\tau)\,d\tau.
		$$
   Further, substituting $r = \mathsf{r}_\infty$ into Theorem \ref{thm2.3} yields (\ref{cap-rV}) again, or equivalently,
		$${\rm Cap}_{p}(K) \leq {\rm Cap}_{p}\big({\bar B}(\mathsf{r}_\infty)\big).$$

       \item If
	$$
	\lambda_k=\frac{\Big(\int_r^\infty \sinh^{\frac{1-n}{p-1}} s\,ds\Big)^p}{
	\int_{\partial K} {\rm p}_k \, d\mu\cdot\int_0^\infty \sinh^{\frac{p(1-n)}{p-1}}(r+\tau)\cdot\cosh^{n-1-k} \tau\cdot  \sinh^k \tau\,d\tau}\ \ \forall\ \ k\in\{0,1,...,n-1\},
$$
then the argument for \eqref{2.7} yields that \eqref{2.8}'s right-hand-side enjoys the following minimizing process
$$
\min\left\{\sum\limits_{k=0}^{n-1} \binom{n-1}{k}\lambda_k^{-1}:\quad \sum_{k=0}^{n-1}\lambda_k=1\right\}=\left(\sum_{i=0}^{n-1}\binom{n-1}{i}^\frac12\right)^2.
$$
	
	\end{itemize}
\end{rem}

\section{Dominating ${\rm Cap}_{1<p\le 2<n}$ through $W_{1}\, |\, W_2$}\label{s4}

In order to reach the desired principle, let us suitably generalize \cite[Theorem 1.2]{JX25} (cf. \cite{vBG}) living in the Euclidean space $\mathbb R^n$.

\begin{itemize}
\item Given a compact domain $K\subset\mathbb{R}^n$, denote by ${\rm Cap}^{Eu}_{p}(K)$ the Euclidean $p$-capacity of $K$, and by ${W}^{Eu}_k(K)$ the Euclidean $k$-th quermassintegral for $k\in\{0, \ldots, n\}$. If $K \subset \mathbb{R}^n$ is a convex body with Lipschitz boundary $\partial K$, then
	\begin{equation}\label{cap-rn}
	{\rm Cap}^{Eu}_{p}(K)\leq n \left(\frac{p-1}{n-p}\right)^{1-p} \big({W}^{Eu}_1(K)\big)^{2-p} \big({W}^{Eu}_2(K)\big)^{p-1}\ \ \text{under}\ \ p\in (1,n),
	\end{equation}
	with equality iff $K$ is a Euclidean ball. Note that \eqref{cap-rn}'s limiting case $p\to 1$ takes the following form (cf. \cite[p.188]{LXZ})
	\begin{equation}
	\label{3.1}
	{\rm Cap}^{Eu}_{1}(K)=nW_1^{Eu}(K).
	\end{equation}
	 Yet, upon defining the $K$'s $j$-th quermassintegral radius $\mathsf{r}_j$ by
	$$
	{W}^{Eu}_j(B(\mathsf{r}_j)) ={W}^{Eu}_j(K),
	 $$
	 we find that (\ref{cap-rn}) is equivalent to
	\begin{equation}\label{3.2}
	\frac{{\rm Cap}^{Eu}_{p}(K)}{{\rm Cap}^{Eu}_{p}\big({\bar B}(\mathsf{r}_2)\big)}\leq \left(\frac{{W}^{Eu}_1(K)}{{W}^{Eu}_1\big({\bar B}(\mathsf{r}_2)\big)}\right)^{2-p}\Leftrightarrow\
    \frac{{\rm Cap}^{Eu}_{p}(K)}{{\rm Cap}^{Eu}_{p}\big(\bar{B}(\mathsf{r}_1)\big)}\leq \left(\frac{{W}^{Eu}_2(K)}{{W}^{Eu}_2\big(\bar{B}(\mathsf{r}_1)\big)}\right)^{p-1}
	\end{equation}

\item Using the inverse mean curvature flow (IMCF) as defined below, we obtain that if $K \subset \mathbb{R}^n\ (n\geq 3)$ is a compact domain with its boundary $\partial K$ being smooth, star-shaped, and mean-convex then (\ref{cap-rn}) holds for ${1<p\leq 2}$, with (\ref{cap-rn})'s equality being valid iff $K$ is a Euclidean ball. Upon sending $p\to 1$ in \eqref{cap-rn}, we get \eqref{3.1}'s weak inequality
\begin{equation}
\label{3.1w}
{\rm Cap}^{Eu}_{1}(K)\le nW_1^{Eu}(K),
\end{equation}
with equality if $K$ is a Euclidean ball.

As a matter of fact, let us consider the following IMCF $$
X:\mathbb{S}^{n-1}\times(0,\infty)\to\mathbb{R}^n\ \ \text{solving}\ \
	\begin{cases}
	\partial_t X = \frac{\mathbf n}{\sigma_1};\\
	X(\mathbb{S}^{n-1},0)=\partial K,
	\end{cases}
$$
	where
	$$
	\begin{cases}
	\mathbf{n}=\text{the outward unit normal of the flow hypersurface $\Sigma_t=X(\mathbb{S}^{n-1},t)$};\\
	(n-1)^{-1}\sigma_1={\rm p}_1=\text{the normalized first mean curvature of $\Sigma_t$}.
	\end{cases}
	$$
	As shown in \cite[Section 3]{LX22} with $p\in (1,\infty)$, we have
	\begin{equation}\label{cap-Tp}
	\begin{cases}
	{\rm Cap}_{p}^{Eu}(K) \leq \left( \int_0^\infty \big(T_p(t)\big)^{\frac{1}{1-p}} \, dt \right)^{1-p};\\
	T_p(t) = \int_{\Sigma_t} \sigma_1^{p-1} \, d\mu_t.
	\end{cases}
	\end{equation}
	Since $1<p\leq 2$, upon applying the H\"older inequality, we get
	\begin{equation}\label{Tp}
	\begin{cases}
	T_p(t) \leq \left(\int_{\Sigma_t} \sigma_1 \, d\mu_t\right)^{p-1} \left(\int_{\Sigma_t}  d\mu_t\right)^{2-p}=\Big(n(n-1){W}^{Eu}_2(K_t)\Big)^{p-1} \big(e^t|\partial K_0|\big)^{2-p};\\
	K_t=\text{the domain enclosed by $\Sigma_t$ with $K_0=K$},
	\end{cases}
	\end{equation}
	where we have used $$|\partial K_t|=|\Sigma_t|=e^t|\partial K_0|.$$ Along the IMCF $\Sigma_t$, the variational formula, together with the Newton-Maclaurin inequality, yields
	\begin{align}\label{dt-dw2}
	\frac{d}{dt} {W}^{Eu}_2(K_t) &= \left(\frac{n-2}{n}\right)\int_{\partial K_t} \left(\frac{\sigma_2}{\binom{n-1}{2} \sigma_1}\right)\, d\mu_t \\
	&\leq \left(\frac{n-2}{n(n-1)^2}\right) \int_{\partial K_t} \sigma_1 \, d\mu_t\notag\\
	&= \left(\frac{n-2}{n-1}\right) {W}^{Eu}_2(K_t).\notag
	\end{align}
	Thus
	
	\begin{equation}\label{W2}
	\begin{cases}
	\frac{d}{dt}\left(e^{-\big(\frac{n-2}{n-1}\big)t}{W}^{Eu}_2(K_t)\right) \leq 0;\\
	{W}^{Eu}_2(K_t)\leq e^{\big(\frac{n-2}{n-1}\big)t}{W}^{Eu}_2(K_0).
	\end{cases}
	\end{equation}
	Via combining (\ref{W2}) with (\ref{cap-Tp})-(\ref{Tp}), we deduce
	\begin{align*}
	{\rm Cap}^{Eu}_{p}(K_0) &\leq \left( \int_0^\infty \big(n(n-1){W}^{Eu}_2(K_t)\big)^{-1} \Big(e^t n{W}^{Eu}_1(K_0)\Big)^{\frac{2-p}{1-p}} \, dt \right)^{1-p}\\
	&= n (n-1)^{p-1} \big({W}^{Eu}_1(K_0)\big)^{2-p}\left( \int_0^\infty \big({W}^{Eu}_2(K_t)\big)^{-1} e^{\frac{2-p}{1-p}t} \, dt \right)^{1-p}\\
	&\leq n (n-1)^{p-1} \big({W}^{Eu}_1(K_0)\big)^{2-p} \big({W}^{Eu}_2(K_0)\big)^{p-1} \left( \int_0^\infty  e^{-\frac{n-p}{(n-1)(p-1)}t} \, dt \right)^{1-p}\\
	&=n \left(\frac{p-1}{n-p}\right)^{1-p} \big({W}^{Eu}_1(K_0)\big)^{2-p} \big({W}^{Eu}_2(K_0)\big)^{p-1},
	\end{align*}
	whose limiting case $p\to 1$ is just
	$$
		{\rm Cap}^{Eu}_{1}(K_0)\le n \big({W}^{Eu}_1(K_0)\big).
		$$
	Moreover, if \eqref{cap-rn}'s equality holds, then by analyzing the equality cases of (\ref{Tp}) and (\ref{dt-dw2}), we see that all points on $\partial K_t$ are umbilical so that $K$ is a Euclidean ball.
	
\end{itemize}

Inspired by not only the above Euclidean case but also \cite[Theorem 3]{LLX}, we turn to the corresponding results in the hyperbolic space $\mathbb{H}^n$ which make the third sharp principle for capacities-to-quermassintegrals as stated below.

\begin{thm}\label{thm3.1} Given $1<p\le 2<n$,
	let $K \subset \mathbb{H}^n$ be a compact domain with its boundary $\partial K$ being smooth, star-shaped, and mean-convex with
	
	\begin{equation}\label{cap-Hn-cor}
	\begin{cases}
	W_1\big(\bar{B}(\mathsf{r}_1)\big)=W_1(K);\\
	W_2\big({\bar B}(\mathsf{r}_2)\big)=W_2(K);\\
	\mathsf{r}_1\leq \mathsf{r}_2\quad\text{ (cf. \cite{BGL19, GL21})}.
	\end{cases}
	\end{equation}

	\begin{itemize}
	\item There holds
	\begin{equation}\label{cap-Hn}
	{{\rm Cap}_{p}(K)}\le{{\rm Cap}_{p}\big({\bar B}(\mathsf{r}_2)\big)} \left(\frac{W_1(K)}{W_1\big({\bar B}(\mathsf{r}_2)\big)}\right)^{2-p},
	\end{equation}
	with equality iff $K$ is a geodesic ball.
	\item There holds
\begin{equation}\label{cap-Hn2}
	{{\rm Cap}_{p}(K)}\le {{\rm Cap}_{p}\big(\bar{B}(\mathsf{r}_1)\big)}\left(\frac{W_2(K)+n^{-1} | {\bar B}(\mathsf{r}_1)|}{W_2\big(\bar{B}(\mathsf{r}_1)\big)+n^{-1} | {\bar B}(\mathsf{r}_1)|}\right)^{p-1},
	\end{equation}
with equality iff $K$ is a geodesic ball.
	
	\end{itemize}
\end{thm}

\begin{proof} A consideration of the following IMCF
	$$X:\mathbb{S}^{n-1}\times(0,\infty)\to\mathbb{H}^n\ \ \text{solving}\ \
	\begin{cases}
	\partial_t X = \frac{\mathbf{n}}{\sigma_1}=\frac{\mathbf{n}}{(n-1)\mathsf{p}_1}; \\
	X(\mathbb{S}^{n-1},0)= \partial K,
	\end{cases}
	$$
	gives (cf. \cite{LLX})
	\begin{equation}\label{cap-Tp-Hn}
	\begin{cases}
	{\rm Cap}_{p}(K) \leq \left( \int_0^\infty \big(T_p(t)\big)^{\frac{1}{1-p}} \, dt \right)^{1-p};\\
	T_p(t) = \int_{\partial K_t} \sigma_1^{p-1} \, d\mu_t=\int_{\partial K_t} \big((n-1){\rm p}_1\big)^{p-1} \, d\mu_t.
	\end{cases}
	\end{equation}
	Thanks to $1<p\leq 2<n$, the H\"older inequality, plus
	$$
	|\partial K_t|=e^t|\partial K_0|,
	$$
	 implies
	\begin{align}\label{Tp-Hn}
	T_p(t)&\leq \left(\int_{\partial K_t} \big((n-1){\rm p}_1\big) \, d\mu_t\right)^{p-1} \left(\int_{\partial K_t}  d\mu_t\right)^{2-p}\\
	&=\left(n(n-1)W_2(K_t)+(n-1)|K_t|\right)^{p-1} \big(e^t|\partial K_0|\big)^{2-p}.\notag
	\end{align}
	\medskip

	Along IMCF, the variational formula together with the Newton-Maclaurin inequality, yields
	\begin{align}\label{dt-dw2-Hn}
	\frac{d}{dt} W_2(K_t) &= \left(\frac{n-2}{n}\right) \int_{\partial K_t} \frac{\sigma_2}{\binom{n-1}{2} \big((n-1){\rm p}_1\big)} \, d\mu_t \\
	&\leq \left(\frac{n-2}{n(n-1)^2}\right)\int_{\partial K_t} \big((n-1){\rm p}_1\big) \, d\mu_t\notag\\
	&= \left(\frac{n-2}{n-1}\right) W_2(K_t)+\left(\frac{n-2}{n(n-1)}\right)|K_t|,\notag
	\end{align}
	where $\sigma_2$ stands for the second mean curvature.
	
	By the result of Gerhardt \cite{Ge11}, the flow hypersurface $\partial K_t$ of IMCF remains star-shaped, mean-convex, expands to infinity, and the principal curvatures $\kappa_i$ decay to $1$ exponentially as $t\to\infty$.
\subsubsection*{\underline{Argument for \eqref{cap-Hn}}}
Via letting
	$$I(s)\equiv f_0\circ f_2^{-1}(s)\ \ \forall\ \ s>0,
	$$ we obtain
	\begin{align*}
	\frac{d}{dt} W_2(K_t)\leq \left(\frac{n-2}{n-1}\right) W_2(K_t)+\left(\frac{n-2}{n(n-1)}\right)I\big(W_2(K_t)\big).
	\end{align*}
	Notice that
	$$r=f_2^{-1}(s)\Rightarrow
	\begin{cases}
	s=W_2\big({\bar B}(r)\big)=n^{-1}\Big(\omega_{n-1}\sinh^{n-2} r \cosh r-|\bar{B}(r)|\Big);\\
	I(s)=|{\bar B}(r)|=\omega_{n-1}\sinh^{n-2} r \cosh r-ns.
	\end{cases}
	$$
	So, upon setting
	$$
	\begin{cases} q(r)\equiv\int_{\partial {\bar B}(r)}{\rm p}_1\, d\mu;\\
	J(s)\equiv q\circ f_2^{-1}(s),
	\end{cases}
	$$ we have
	\begin{align}\label{dt-dw2-Hn-2}
	\frac{d}{dt} W_2(K_t)\leq \left(\frac{n-2}{n(n-1)}\right) J\big(W_2(K_t)\big).
	\end{align}
	
	Fixed $s_0>0$, define
	\begin{align*}
	\zeta(s)\equiv\int_{s_0}^s\frac{d\tau}{\left(\frac{n-2}{n(n-1)}\right)J(\tau)}\quad\forall\ \  s>0.
	\end{align*}
	It follows from (\ref{dt-dw2-Hn-2}) that
	\begin{align}\label{dt-dw2-Hn-3}
	\begin{cases}
	\frac{d}{dt}(\zeta\big(W_2(K_t)\big)-t)=\frac{\frac{d}{dt} W_2(K_t)}{\left(\frac{n-2}{n(n-1)}\right) J\big(W_2(K_t)\big)}-1\leq0;\\
	W_2(K_t)\leq \zeta^{-1}(t+\zeta(W_2(K_0))).
	\end{cases}
	\end{align}
	Substituting the above inequality into (\ref{Tp-Hn}), we obtain
	\begin{align*}
	T_p(t) &\leq \left((n-1)J\big(W_2(K_t)\big)\right)^{p-1} (e^t nW_1(K_0))^{2-p}\\
	&\leq (n-1)^{p-1} (nW_1(K_0))^{2-p} e^{(2-p)t}  \left(J\circ  \zeta^{-1}(t+\zeta(W_2(K_0)))\right)^{p-1},
	\end{align*}
whence
	\begin{align}\label{1.3a}
	{\rm Cap}_{p}(K)
	\leq (n-1)^{p-1} (nW_1(K_0))^{2-p} \left( \int_0^\infty \frac{e^{\frac{2-p}{1-p} t} \, dt}{J\circ  \zeta^{-1}(t+\zeta(W_2(K_0)))} \right)^{1-p}.
	\end{align}
	
In order to simplify the function $\zeta$, we consider the evolution of $\partial {\bar B}(r_0)$ along IMCF: the flow hypersurface $\partial ({\bar B}(r_0))_t$ not only remains a geodesic sphere $\partial  {\bar B}(r_t)$ but also satisfies $$|\partial  {\bar B}(r_t)|=e^t|\partial {\bar B}(r_0)|.$$ Then
	$$\sinh r_t=e^{\frac{t}{n-1}}\sinh r_0.$$
	Since the geodesic sphere achieves not only (\ref{dt-dw2-Hn})'s equality but also (\ref{dt-dw2-Hn-3})'s equality, we obtain
	\begin{align*}
	W_2( {\bar B}(r_t))= \zeta^{-1}(t+\zeta(W_2({\bar B}(r_0)))).
	\end{align*}
	Via taking $$r_0\equiv f_2^{-1}(W_2(K_0)) - \text{i.e.} - W_2({\bar B}(r_0))=W_2(K_0),
	$$
	we get not only
	\begin{align*}
	f_2(r_t)=\zeta^{-1}(t+\zeta(W_2(K))),
	\end{align*}
	but also
	\begin{align*}
	J\circ  \zeta^{-1}(t+\zeta(W_2(K_0)))
	&=q(r_t)\\
	&=\omega_{n-1}(\sinh^{n-2} r_t)(\cosh r_t)\\
	&=\omega_{n-1} e^{\big(\frac{n-2}{n-1}\big)t}(\sinh^{n-2}r_0) \sqrt{e^{\frac{2t}{n-1}}\sinh^2 r_0+1}.
	\end{align*}
	Substituting the above into the integral in (\ref{1.3}) yields
	\begin{align*}
	\int_0^\infty \frac{e^{\frac{2-p}{1-p} t} \, dt}{J\circ  \zeta^{-1}(t+\zeta(W_2(K_0)))}
	&=\int_0^\infty \frac{e^{\frac{2-p}{1-p} t} \, dt}{\omega_{n-1} e^{\big(\frac{n-2}{n-1}\big)t} \sinh^{n-2}r_0 \sqrt{e^{\frac{2t}{n-1}}\sinh^2 r_0+1}} \\
	&=\Big({\omega_{n-1}\sinh^{n-2} r_0}\Big)^{-1} \int_0^\infty \frac{e^{-\frac{n-p}{(n-1)(p-1)} t} \, dt}{ \sqrt{e^{\frac{2t}{n-1}}\sinh^2 r_0+1}} \\
	&=\left(\frac{n-1}{\omega_{n-1}}\right)(\sinh r_0)^{\frac{(n-1)(2-p)}{p-1}} \int_{r_0}^\infty (\sinh \alpha)^{-\frac{n-1}{p-1}}d\alpha,
	\end{align*}
	where we have used the change of variables $$\sinh \alpha=(\sinh r_0) e^{\frac{t}{n-1}}.$$
	
	By (\ref{1.3}), we obtain
	\begin{align*}
&	{\rm Cap}_{p}(K)\\
	&\ \leq (n-1)^{p-1} (nW_1(K_0))^{2-p} \left( \left(\frac{n-1}{\omega_{n-1}}\right)(\sinh r_0)^{\frac{(n-1)(2-p)}{p-1}} \int_{r_0}^\infty (\sinh \alpha)^{-\frac{n-1}{p-1}}d\alpha \right)^{1-p}\\
	&\ =  \left(\frac{W_1(K_0)}{\frac{\omega_{n-1}}{n}\sinh^{n-1}r_0}  \right)^{2-p} \left( \int_{r_0}^\infty (\omega_{n-1}\sinh^{n-1} \alpha)^{\frac{1}{1-p}}d\alpha \right)^{1-p}\\
	&\ = \left(\frac{W_1(K_0)}{f_1\circ f_2^{-1}(W_2(K_0))}  \right)^{2-p} g_p\circ f_2^{-1}(W_2(K_0)).
	\end{align*}
	Now, upon noting $r_0=\mathsf{r}_2$, we obtain (\ref{cap-Hn}).
	
Last but not least, if (\ref{cap-Hn}) or (\ref{cap-Hn2})'s equality holds, then by analyzing the equality cases of (\ref{Tp-Hn}) \& (\ref{dt-dw2-Hn}), we see that all points on $\partial K_t$ are umbilical. Hence, $K$ is a geodesic ball. Conversely, if $K$ is a geodesic ball, then (\ref{cap-Hn})'s equality is automatically true.

\subsubsection*{\underline{Argument for \eqref{cap-Hn2}}}

We start the proof again from \eqref{dt-dw2-Hn}. A direct integration indicates
\begin{equation}\label{eq417}
W_2(K_t)\leq e^{\big(\frac{n-2}{n-1}\big)t}\left(W_2(K_0)+\bigg(\frac{n-2}{n(n-1)}\bigg)\int_0^t e^{-\big(\frac{n-2}{n-1}\big)\tau}|K_\tau|\,d\tau\right)
\end{equation}
Let $r_\tau$ be the area radius of $\partial K_\tau$ in the sense of $$|\partial K_\tau|=|\partial {\bar B}(r_\tau)|.$$ Because of
$$|\partial K_\tau|=e^\tau|\partial K_0|=e^\tau|\partial  {\bar B}(\mathsf{r}_1)|,$$ we get
$$
\sinh r_\tau=e^{\frac{\tau}{n-1}}\sinh \mathsf{r}_1\ \ \&\ \ r_0=\mathsf{r}_1.
$$
Upon recalling the isoperimetric inequality in $\mathbb H^n$ as seen below
$$
|K_\tau|\leq | {\bar B}(r_\tau)|,
$$ we derive
\begin{align}\label{eq418}
  &\int_0^t e^{-\big(\frac{n-2}{n-1}\big)\tau}|K_\tau|\,d\tau\\
   &\ \ \leq\int_0^t e^{-\big(\frac{n-2}{n-1}\big)\tau} | {\bar B}(r_\tau)|\,d\tau\notag\\ &\ \ =(n-1)\int_{\mathsf{r}_1}^{r_t}\bigg(\frac{\sinh \mathsf{r}_1}{\sinh y}\bigg)^{n-2}| {\bar B}(y)|\coth y\,dy\notag
  \\
  &\ \ =\left(\frac{n-1}{2-n}\right)(\sinh \mathsf{r}_1)^{n-2}\left(| {\bar B}(y)|(\sinh y)^{2-n}\Big|_{\mathsf{r}_1}^{r_t}-\int_{\mathsf{r}_1}^{r_t}(\sinh y)^{2-n}\omega_{n-1}(\sinh y)^{n-1}\,dy\right)\notag
  \\ &\ \ =\left(\frac{n-1}{n-2}\right)\left(| {\bar B}(\mathsf{r}_1)|-\bigg(\frac{\sinh \mathsf{r}_1}{\sinh r_t}\bigg)^{n-2}| {\bar B}(r_t)|+\omega_{n-1}(\sinh \mathsf{r}_1)^{n-2}(\cosh r_t-\cosh \mathsf{r}_1)\right)\notag
  \\ &\ \  \equiv\left(\frac{n-1}{n-2}\right) H(r_t).\notag
  \end{align}
Via changing the variables
$$\sinh x=e^{\frac t{n-1}}\sinh \mathsf{r}_1\ \ (\text{or}\ x=r_t),$$ we have
\begin{align}\label{eq419}
& \int_0^\infty \big(T_p(t)\big)^{\frac{1}{1-p}}\,dt\\
  &\ \  \geq \int_0^\infty\left( \frac{(|\partial K_t|)^{\frac{2-p}{1-p}}}{n(n-1)W_2(K_t)+(n-1)|K_t|}\right)\,dt\notag
  \\
  &\ \ \geq\int_0^\infty \left(\frac{(|\partial  {\bar B}(r_t)|)^{\frac{2-p}{1-p}}}{n(n-1)e^{\big(\frac{n-2}{n-1}\big)t}W_2(K_0)+(n-1)e^{\big(\frac{n-2}{n-1}\big)t}H(r_t)+(n-1)| {\bar B}(r_t)|}\right)\,dt\notag
  \\
  &\ \ =\int_{\mathsf{r}_1}^\infty \left(\frac{|\partial  {\bar B}(x)|^{\frac{2-p}{1-p}}\cdot (n-1)\coth x}{n(n-1)\Big(\frac{\sinh x}{\sinh \mathsf{r}_1}\Big)^{n-2}W_2(K_0)+(n-1)\Big(\frac{\sinh x}{\sinh \mathsf{r}_1}\Big)^{n-2}H(x)+(n-1)| {\bar B}(x)|}\right)\,dx\notag
  \\
  &\ \ =\int_{\mathsf{r}_1}^\infty \left(\frac{|\partial  {\bar B}(x)|^{\frac{2-p}{1-p}}\cdot\coth x}{(\sinh x)^{n-2} \left(\frac{nW_2(K_0)+| {\bar B}(\mathsf{r}_1)|}{(\sinh \mathsf{r}_1)^{n-2}} + \omega_{n-1} (\cosh x - \cosh \mathsf{r}_1) \right)}\right)\,dx\notag
  \\
  &\ \ = \int_{\mathsf{r}_1}^\infty\left(\frac{(\omega_{n-1})^{\frac{1}{1-p}}(\sinh x)^{\frac{n-1}{1-p}}}{1+\frac{nW_2(K_0)-nW_2\big(\bar{B}(\mathsf{r}_1)\big)}{\omega_{n-1}(\sinh \mathsf{r}_1)^{n-2}\cosh x}}\right)\,dx, \notag
  \end{align}
where we have used $$
W_2\big(\bar{B}(\mathsf{r}_1)\big)=n^{-1}\int_{\partial  {\bar B}(\mathsf{r}_1)}{\rm p}_1\,d\mu-n^{-1}W_0\big(\bar{B}(\mathsf{r}_1)\big)=\frac{\omega_{n-1}}{n}(\sinh \mathsf{r}_1)^{n-2}\cosh \mathsf{r}_1-n^{-1}| {\bar B}(\mathsf{r}_1)|.
$$

Finally, due to \eqref{cap-Hn-cor} we get
$$
\begin{cases}
	\cosh x\geq \cosh \mathsf{r}_1;\\ W_2(K_0)=W_2\big({\bar B}(\mathsf{r}_2)\big)\geq W_2\big(\bar{B}(\mathsf{r}_1)\big),
\end{cases}
$$
thereby deriving
\begin{align*}
{\rm Cap}_{p}(K) & \leq \left( \int_0^\infty \big(T_p(t)\big)^{\frac{1}{1-p}} \, dt \right)^{1-p}
\\ & \leq \left( \int_{\mathsf{r}_1}^\infty \left(\frac{(\omega_{n-1})^{\frac{1}{1-p}}(\sinh x)^{\frac{n-1}{1-p}}}{1+\frac{nW_2(K_0)-nW_2\big(\bar{B}(\mathsf{r}_1)\big)}{\omega_{n-1}(\sinh \mathsf{r}_1)^{n-2}(\cosh \mathsf{r}_1)}}\right)\,dx \right)^{1-p}\notag
\\ &={\rm Cap}_{p}\big(\bar{B}(\mathsf{r}_1)\big)\left(1+\frac{nW_2(K_0)-nW_2\big(\bar{B}(\mathsf{r}_1)\big)}{\omega_{n-1}(\sinh \mathsf{r}_1)^{n-2}(\cosh \mathsf{r}_1)}\right)^{p-1}\notag
\\ &={\rm Cap}_{p}\big(\bar{B}(\mathsf{r}_1)\big)\left(\frac{nW_2(K_0)+| {\bar B}(\mathsf{r}_1)|}{nW_2\big(\bar{B}(\mathsf{r}_1)\big)+| {\bar B}(\mathsf{r}_1)|}\right)^{p-1},\notag
\end{align*}
as desired in (\ref{cap-Hn2}) whose equality case can be checked in a way similar to verifying (\ref{cap-Hn})'s equality case.

\end{proof}

\begin{rem} Two comments are in order.
\begin{itemize}
\item Just like \eqref{cap-rn} which can be treated as a Euclidean capacity uncertainty principle, a multiplication of \eqref{cap-Hn} \& \eqref{cap-Hn2} produces the following hyperbolic capacity uncertainty principle

\begin{equation}\label{4.21b}
\frac{\big({\rm Cap}_{p}(K)\big)^2}{{{\rm Cap}_{p}\big({\bar B}(\mathsf{r}_2)\big)}    {{\rm Cap}_{p}\big({\bar B}(\mathsf{r}_1)\big)}   }\le \left(\frac{W_1(K)}{W_1\big({\bar B}(\mathsf{r}_2)\big)}\right)^{2-p} \left(\frac{W_2(K)+n^{-1} | {\bar B}(\mathsf{r}_1)|}{W_2\big(\bar{B}(\mathsf{r}_1)\big)+n^{-1} | {\bar B}(\mathsf{r}_1)|}\right)^{p-1},
\end{equation}
with equality if $K$ is a geodesic ball.

\item Upon letting $p\to 1$ in \eqref{cap-Hn2} or \eqref{4.21b}, we find
	\begin{equation}\label{cap-Hn2w}
	{{\rm Cap}_{1}(K)}\le{{\rm Cap}_{1}\big(\bar{B}(\mathsf{r}_1)\big)}=nW_1\big(\bar{B}(\mathsf{r}_1)\big)=nW_1(K)=|\partial K|,
	\end{equation}
	with its inequality becoming an equality whenever $K$ is a geodesic ball in $\mathbb H^n$. Clearly, \eqref{cap-Hn2w} nicely corresponds to \eqref{3.1w}.
\end{itemize}

\end{rem}

\section{Capacity bounds through $W_1$, $W_2+n^{-1}W_0$, and effective curvature radii}\label{s5}

 \cite{LLX} contains two important inequalities on ${\rm Cap}_{p\in (1,3]}(K)$ for any compact domain $K\subset\mathbb H^n$ with its boundary $\partial K$ being smooth, star-shaped, and mean-convex as explained below.

\begin{itemize}
\item If $p=2$ \& $\mathsf I$ is the isoperimetric function defined by
$$|{\bar B}(r)| = {\mathsf I}\big(|\partial{\bar B}(r)|\big),
$$
then

		\begin{align}\label{cap-2}
		{\rm Cap}_2(K)
		\leq  \frac{n(n-1)}{ \left(\int_0^\infty \Biggl( e^{\big(\frac{n-2}{n-1}\big)t} \biggl( W_2(K) + \frac{\int_0^t e^{-\big(\frac{n-2}{n-1}\big)\tau}{\mathsf I}(e^\tau |\partial K|)  d\tau}{\left(\frac{n-2}{n(n-1)}\right)^{-1}}\biggr) + n^{-1}{\mathsf I}(e^t |\partial K|) \Biggr)^{-1} dt\right)}.
		\end{align} 

Below are some observations on \eqref{cap-2}.

\begin{itemize}
\item If $n=2$, \eqref{cap-2} reduces to
\begin{equation}\label{cap2-H2}
{\rm Cap}_2(K) \leq \frac{2\pi}{\text{arsinh}(2\pi |\partial K|^{-1})},
\end{equation}
with equality iff $K$ is a geodesic ball.
In fact, (\ref{cap2-H2}) is equivalent to
\begin{equation}\label{cap-Hn2a-new}
	{{\rm Cap}_{2}(K)}\le {{\rm Cap}_{2}\big(\bar{B}(\mathsf{r}_1)\big)},
\end{equation}
where $W_1(\bar{B}(\mathsf{r}_1))=W_1(K)$.
For $n\ge 3$, however, the sharp $W_1$-comparison \eqref{cap-Hn2a-new} is not available at present. From \eqref{cap-Hn2} one only obtains the weaker estimate
\begin{equation}\label{cap-Hn2a-high}
	\frac{{{\rm Cap}_{2}(K)}}{{{\rm Cap}_{2}\big(\bar{B}(\mathsf{r}_1)\big)}}  \le
	\frac{W_2(K)+n^{-1} | {\bar B}(\mathsf{r}_1)|}
	{W_2\big(\bar{B}(\mathsf{r}_1)\big)+n^{-1} | {\bar B}(\mathsf{r}_1)|}.
\end{equation}

\item Upon recalling that for any smooth, star-shaped, and mean-convex hypersurface $\partial K$ in $\mathbb{H}^{n}$ there holds
\begin{align*}
W_1(K)=n^{-1}|\partial K|\underset{R}{\lesssim} W_2(K),
\end{align*}
we use (\ref{cap-2}) to conclude
\begin{align}
{\rm Cap}_{2}(K) \underset{R}{\lesssim} W_2(K)\quad\text{under}\quad n\geq 3.
\end{align}
\end{itemize}

\item If $1<p\le 3=n$, then
\begin{equation}\label{333}{\tiny
{\rm Cap}_p(K)\leq\frac{ (16\pi)^{\frac{p-1}{2}}|\partial K|^{\frac{3-p}{2}}}{\left( \int_0^\infty \left( \Bigl( \int_{\partial K} ({\rm p}_1^2-1) \, d\mu - 1 \Bigr) e^{-t} + (4\pi)^{-1}{|\partial K| e^{t}} + 1 \right)^{-\frac{1}{2}} \frac{e^{-\big(\frac{3-p}{2(p-1)}\big)t}}{4\pi} \, dt \right)^{p-1}},}
\end{equation}
with equality iff $K$ is a geodesic ball.

Below are two observations around \eqref{333}.

\begin{itemize}
 \item If $p=2<3=n$,  
\begin{equation}\label{3333}{\tiny
\text{Cap}_2(K) \leq \sqrt{16\pi |\partial K|} \left( \int_0^\infty \left( (4\pi)^{-1}\int_{\partial K} ({\rm p}_1^2-1) \, d\mu - 1  + (4\pi)^{-1}{|\partial K| e^{2t}}+ e^t \right)^{-\frac{1}{2}}  \, dt \right)^{-1},}
\end{equation}
with equality iff  $K$ is a geodesic ball.
\item If $n\ge 3$, then from \cite{Hu18} it follows that not only any star-shaped and mean-convex hypersurface $\Sigma\subset\mathbb{H}^n$ satisfies

\begin{equation}\label{will-Hn}
\int_{\Sigma} ({\rm p}_1^2 - 1) \, d\mu \geq \omega_{n-1}^{\frac{2}{n-1}} |\Sigma|^{\frac{n-3}{n-1}},
\end{equation}
with equality iff $\Sigma$ is a geodesic sphere, but also the squared mean curvature difference
\begin{equation}\label{Q}
Q(t)=|\Sigma_t|^{-\frac{n-3}{n-1}}\int_{\Sigma_t} ({\rm p}_1^2 - 1) \, d\mu_t
\end{equation}
is decreasing along IMCF.
\end{itemize}

\end{itemize}

The foregoing discussion, together with the effective curvature radii introduced below and the identities
$$
\begin{cases}
W_1(K)=n^{-1}|\partial K|;\\
n^{-1}\int_{\partial K}{\rm p}_1\,d\mu=W_2(K)+n^{-1}W_0(K),
\end{cases}
$$
leads to the fourth family of sharp capacity comparisons. Besides the quermassintegral bounds, the resulting estimates admit a geometric formulation in terms of the capacity-to-area ratios of geodesic balls; the case $n=3=p+1$ is of independent interest in the physical hyperbolic space $\mathbb H^3$.

\begin{thm}\label{thm3.3} Let $m\in\mathbb N$ and let $K$ be a compact subdomain of $\mathbb H^{n\ge 3}$.
\begin{itemize}
\item If $\partial K$ is smooth horospherically convex with
\begin{equation*}
	\begin{cases}
	\varepsilon=\Big({\int_{\partial K}(1+{\rm p}_1)\, d\mu}\Big)^{-1}{\sqrt{\left|\left(\int_{\partial K}{\rm p}_1\, d\mu\right)^2-|\partial K|\int_{\partial K}{\rm p}_2\, d\mu\right|}};\\
	{\rm RCap}_p(K)=r\ \ \text{solves}\ \ {\rm Cap}_p(K)={\rm Cap}_p\big(\bar{B}(r)\big),
	\end{cases}
	\end{equation*}
	then
	\begin{equation}\label{cap2-H3}
\begin{cases}
	{\rm Cap}_{2}(K) \leq  \left(\frac{2\varepsilon}{\ln\frac{1+\varepsilon}{1-\varepsilon}}\right)\int_{\partial K}(1+{\rm p}_1)d\mu&\text{as}\ \ p=2=n-1;\\
	{\rm RCap}_p(K)\le(n-1)^{-1}\ln\left(\omega_{n-1}^{-1}\int_{ \partial K}\Big(1+{\rm p}_1\Big)^{n-1}\,d\mu\right)&\text{as}\ \ 1<p<\infty,
	\end{cases}
	\end{equation}
	with equality iff $K$ is a geodesic ball.

	\item Suppose that $\partial K$ is smooth, star-shaped, and mean-convex. For every $q\ge 2$, define the $L^q$-mean-curvature radius $\rho_q(K)\in(0,\infty)$ by
	$$
	\frac{\displaystyle \int_{\partial \bar B(\rho_q(K))}{\rm p}_1^q\,d\mu}{|\partial \bar B(\rho_q(K))|}
	=
	\frac{\displaystyle \int_{\partial K}{\rm p}_1^q\,d\mu}{|\partial K|}.
	$$
	If $p\in [3,\infty)$, then
	\begin{equation}\label{cap-will-Hn}
	\frac{{\rm Cap}_{p}(K)}{|\partial K|}
	\le
	\frac{{\rm Cap}_{p}\big(\bar B(\rho_{p-1}(K))\big)}
	{|\partial \bar B(\rho_{p-1}(K))|}
	<
	\frac{\displaystyle \int_{\partial K}\left(\frac{\sigma_1}{p-1}\right)^{p-1}\,d\mu}{|\partial K|},
	\end{equation}
	with equality in the first inequality iff $K$ is a geodesic ball.

	\item Suppose that $\partial K$ is smooth, star-shaped, and h-convex. For every $q\ge 2$, define the $q$-curvature-excess radius $\widehat{\rho}_q(K)\in(0,\infty)$ by
	$$
	\frac{\displaystyle \int_{\partial \bar B(\widehat{\rho}_q(K))}({\rm p}_1^2-1)^{q/2}\,d\mu}{|\partial \bar B(\widehat{\rho}_q(K))|}
	=
	\frac{\displaystyle \int_{\partial K}({\rm p}_1^2-1)^{q/2}\,d\mu}{|\partial K|}.
	$$
	\begin{itemize}
	\item If $1<p\le 2m+1$, then
	\begin{equation}\label{cap-will-Hn2}
\frac{{\rm Cap}_{p}(K)}{|\partial K|}
\le
\frac{{\rm Cap}_{p}\big(\bar B(\widehat{\rho}_{2m}(K))\big)}
{|\partial\bar B(\widehat{\rho}_{2m}(K))|}.
\end{equation}
with equality iff $K$ is a geodesic ball. Moreover, if $m=1$, then the condition ``h-convex'' can be weakened to ``mean-convex'' and
	$$
	{\rm Cap}_{p}(K)
	\le
	|\partial K|^{\frac{3-p}{2}}
	\left(
	\int_{\partial K}\Big(\frac{\sigma_1}{p-1}\Big)^2\,d\mu
	\right)^{\frac{p-1}{2}}.
	$$
\item As a consequence of \eqref{cap-will-Hn}, if $p>2m+1\ge3$, set $\|{\rm p}_1\|_\infty=\max_{\partial K}{\rm p}_1$ and define the interpolating curvature-excess radius $\rho_{p,m}(K)\in(0,\infty)$ by
\[
\frac{\displaystyle \int_{\partial \bar B(\rho_{p,m}(K))}({\rm p}_1^2-1)^{\frac{p-1}{2}}\,d\mu}{|\partial \bar B(\rho_{p,m}(K))|}
=
\frac{\displaystyle \int_{\partial K}
\big(\|{\rm p}_1\|_\infty^2-1\big)^{\frac{p-1-2m}{2}}
({\rm p}_1^2-1)^m\,d\mu}{|\partial K|}.
\]
Then
\begin{equation}\label{4.11c}
\frac{{\rm Cap}_{p}(K)}{|\partial K|}
\le
\frac{{\rm Cap}_{p}\big(\bar B(\rho_{p,m}(K))\big)}
{|\partial\bar B(\rho_{p,m}(K))|},
\end{equation}
with equality iff $K$ is a geodesic ball.
	\end{itemize}

\par
	\end{itemize}
\end{thm}
\begin{proof} We are about to validate results one-by-one.

\subsubsection*{\underline{Argument for \eqref{cap2-H3}}}
Two situations are handled below.
\begin{itemize}
\item If $n-3=0=p-2$, then $\varepsilon$'s formula, along with a direct calculation of (\ref{1.3}), yields that
\begin{equation}\label{3.28}
{\rm Cap}_{2}(K) \leq
\begin{cases}
 \left(\frac{2\varepsilon}{\ln\frac{1+\varepsilon}{1-\varepsilon}}\right)\int_{\partial K}(1+{\rm p}_1)\, d\mu&\text{as}\ \ \left(\int_{\partial K}{\rm p}_1\, d\mu\right)^2>|\partial K|\int_{\partial K}{\rm p}_2\, d\mu;\\
\int_{\partial K}(1+{\rm p}_1)\, d\mu&\text{as}\ \ \left(\int_{\partial K}{\rm p}_1\, d\mu\right)^2=|\partial K|\int_{\partial K}{\rm p}_2\, d\mu;\\
\left(\frac{\varepsilon}{\arctan \varepsilon}\right)\int_{\partial K}(1+{\rm p}_1)\, d\mu&\text{as}\ \ \left(\int_{\partial K}{\rm p}_1\, d\mu\right)^2<|\partial K|\int_{\partial K}{\rm p}_2\, d\mu,
\end{cases}
\end{equation}
is valid for any compact domain $K\subset\mathbb{H}^3$ with smooth convex boundary $\partial K$. Furthermore, when $K\subset\mathbb{H}^3$ is horospherically convex, we have $$\int_{\partial K} {\rm p}_2\, d\mu-|\partial K|=4\pi\ \ \text{ within}\ \  \mathbb{H}^3,
$$
thereby employing
the hyperbolic Alexandrov-Fenchel inequality (cf.~\cite{GWW14,LWX14,WX14}) to gain
\begin{equation}\label{3.29}
\left(\int_{\partial K}{\rm p}_1\, d\mu\right)^2\geq |\partial K|^2+4\pi|\partial K|
=|\partial K|\int_{\partial K}{\rm p}_2\, d\mu,
\end{equation}
with equality iff $K$ is a geodesic ball. Consequently, a combination of \eqref{3.28}-\eqref{3.29} derives \eqref{cap2-H3}'s first inequality. Of course, \eqref{cap2-H3}'s equality case under $p=2=n-1$ amounts to \eqref{3.29}'s equality case with $p=2=n-1$. Thus this equality case is true iff $K$ is a geodesic ball.

\item Suppose $1<p<\infty$. According to \cite[Theorem 5.1]{JX}, there holds (cf. \eqref{cap-rV})
\begin{equation}
\label{5.13a}
{\rm RCap}_p(K)\le ({n-1})^{-1}\ln{\rm RV}(K)
\end{equation}
with equality iff $K$ is a geodesic ball. In the meantime, \cite[Theorem 1.3]{JY} for $\mathbb H^n$ especially ensures
\begin{equation}
\label{5.13b}
\int_{ \partial K}\left(1+{\rm p}_1\right)^{n-1}\,d\mu\ge \omega_{n-1}{\rm RV}(K)
\end{equation}
with equality iff $K$ is a geodesic ball. Now, a combination of \eqref{5.13a}-\eqref{5.13b}, together with their equality settings, implies \eqref{cap2-H3}'s second inequality with its equality.
\end{itemize}

\subsubsection*{\underline{Argument for \eqref{cap-will-Hn}}}

Since ${\rm p}_1\equiv\coth r$ on the geodesic sphere $\partial\bar B(r)$, the defining identity for the $L^q$-mean-curvature radius is equivalently
\[
\begin{cases}
\coth\rho_q(K)=\left(\frac{\displaystyle \int_{\partial K}{\rm p}_1^q\,d\mu}{|\partial K|}\right)^{\frac1q};\\
\frac1{\sinh^2\rho_q(K)}=
\left(\frac{\displaystyle \int_{\partial K}{\rm p}_1^q\,d\mu}{|\partial K|}\right)^{\frac2q}-1.
\end{cases}
\]

A consideration of the IMCF
$$X:\mathbb{S}^{n-1}\times(0,\infty)\to\mathbb{H}^n\ \ \text{solving}\ \
\begin{cases}
	\partial_t X = \frac{\mathbf{n}}{(n-1){\rm p}_1}; \\
	X(\mathbb{S}^{n-1},0)= \partial K,
\end{cases}
$$
along with the result of Gerhardt \cite{Ge11}, gives that not only the flow hypersurface $\partial K_t$ of IMCF remains star-shaped, mean-convex, expands to infinity, and the principal curvatures $\kappa_i$ decay to $1$ exponentially as $t\to\infty$, but also
\begin{equation}\label{cap-Tp-Hn-s5}
\begin{cases}
{\rm Cap}_{p}(K) \leq \left( \int_0^\infty \big(T_p(t)\big)^{\frac{1}{1-p}} \, dt \right)^{1-p};\\
T_p(t) = \int_{\partial K_t} \big((n-1){\rm p}_1\big)^{p-1} \, d\mu_t;\\
|\partial K_t|=e^t|\partial K|.
\end{cases}
\end{equation}

 If $3\le p<\infty$, then via differentiating \eqref{cap-Tp-Hn-s5}'s $T_p(t)$ with respect to $t$, we are about to employ not only the basic inequality
$$
(n-1){\rm p}_1^2\le |\mathsf{h}|^2 = \text{the squared norm of the second fundamental form $\mathsf{h}=({\mathsf{h}}_{ij})$ of $\mathbb H^n$}
$$
but also the given condition $p\ge 3$ to obtain that if
$$
T_{p,\ast}(t)=\int_{\partial K_t}{\rm p}_1^{p-1}\,d\mu_t=(n-1)^{1-p}T_p(t)
$$
then
{\small
\begin{align*}
\frac{d T_{p,\ast}(t)}{dt}&=\int_{\partial K_t}{\rm p}_1^{p-1}\,d\mu_t+(p-1)\int_{ \partial K_t}{\rm p}_1^{p-2}\left(\frac{\Delta\sigma_1^{-1}}{1-n}+{|{\mathsf h}|^2}{\big((1-n)\sigma_1\big)^{-1}}+\sigma_1^{-1}\right)\,d\mu_t\\
&\le \left(\frac{(1-p)(p-2)}{(n-1)^2}\right)\int_{\partial K_t}{\rm p}_1^{p-4}|\nabla{\rm p}_1|^2\,d\mu_t+\frac{\int_{\partial K_t}{\rm p}_1^{p-1}\,d\mu_t}{\frac{n-1}{n-p}}+\frac{\int_{\partial K_t}{\rm p}_1^{p-3}\,d\mu_t}{\frac{n-1}{p-1}}\\
&\le\left(\frac{n-p}{n-1}\right)\int_{\partial K_t}{\rm p}_1^{p-1}\,d\mu_t+\left(\frac{p-1}{n-1}\right)\int_{\partial K_t}{\rm p}_1^{p-3}\,d\mu_t\\
&\ \ \le\left(\frac{n-p}{n-1}\right)T_{p,\ast}(t)+\left(\frac{p-1}{n-1}\right)|\partial K_t|^\frac{2}{p-1}\big(T_{p,\ast}(t)\big)^\frac{p-3}{p-1},
\end{align*}}
where $\sigma_1=(n-1){\rm p}_1$ still stands for the first mean curvature without normalization. Consequently, we achieve
\begin{equation}
\label{4.19b}
\begin{cases}
\frac{d}{dt}\left(\frac{T_{p,\ast}(t)}{|\partial K_t|^\frac{n-p}{n-1}}\right)\le\left(\frac{p-1}{n-1}\right)|\partial K_t|^\frac{2}{n-1}\left(|\partial K_t|^\frac{p-n}{n-1}T_{p,\ast}(t)\right)^\frac{p-3}{p-1};\\
\frac{d}{dt}\left(\left(\frac{T_{p,\ast}(t)}{|\partial K_t|^\frac{n-p}{n-1}}\right)^\frac{2}{p-1}-|\partial K_t|^\frac{2}{n-1}\right)\le 0;\\
\big(T_{p,\ast}(t)\big)^\frac1{p-1}\le|\partial K_t|^\frac{n-p}{(n-1)(p-1)}\left(|\partial K_t|^\frac{2}{n-1}\Big(|\partial K|^\frac{p-n}{n-1}T_{p,\ast}(0)\Big)^\frac{2}{p-1}-|\partial K|^\frac{2}{n-1}\right)^{\frac12}.
\end{cases}
\end{equation}
This last estimation, along with the change of variables $s=e^\frac{t}{1-p}$, derives
\begin{align*}
\frac{{\rm Cap}_p(K)}{|\partial K|}
&\le
\left(\frac{n-1}{p-1}\right)^{p-1}
\left(
\int_0^1
\left(
1+
\left[
\left(\frac{\displaystyle \int_{\partial K}{\rm p}_1^{p-1}\,d\mu}{|\partial K|}\right)^{\frac{2}{p-1}}-1
\right]
s^{\frac{2(p-1)}{n-1}}
\right)^{-\frac12}ds
\right)^{1-p}
\\
&=
\left(\frac{n-1}{p-1}\right)^{p-1}
\left(
\int_0^1
\left(
1+\frac{s^{\frac{2(p-1)}{n-1}}}{\sinh^2\rho_{p-1}(K)}
\right)^{-\frac12}ds
\right)^{1-p}.
\end{align*}
To identify the last expression with the geodesic-ball model, use \eqref{1.1} and make the change of variables
\[
\sinh r=\sinh\rho_{p-1}(K)\,s^{-\frac{p-1}{n-1}}\  \  \forall\  \ s\in (0,1].
\]
Then
\[
\int_{\rho_{p-1}(K)}^\infty
\left(
\frac{\sinh\rho_{p-1}(K)}{\sinh r}
\right)^{\frac{n-1}{p-1}}dr
=
\frac{p-1}{n-1}
\int_0^1
\left(
1+\frac{s^{\frac{2(p-1)}{n-1}}}{\sinh^2\rho_{p-1}(K)}
\right)^{-\frac12}ds,
\]
and hence
\[
\left(\frac{n-1}{p-1}\right)^{p-1}
\left(
\int_0^1
\left(
1+\frac{s^{\frac{2(p-1)}{n-1}}}{\sinh^2\rho_{p-1}(K)}
\right)^{-\frac12}ds
\right)^{1-p}
=
\frac{{\rm Cap}_{p}\big(\bar B(\rho_{p-1}(K))\big)}
{|\partial\bar B(\rho_{p-1}(K))|}.
\]
Moreover, since $0<s<1$,
\[
\int_0^1
\left(
1+\frac{s^{\frac{2(p-1)}{n-1}}}{\sinh^2\rho_{p-1}(K)}
\right)^{-\frac12}ds
>
\left(
1+\frac{1}{\sinh^2\rho_{p-1}(K)}
\right)^{-\frac12},
\]
which yields the strict second inequality in \eqref{cap-will-Hn}. Thus \eqref{cap-will-Hn} follows for $p\in[3,\infty)$.

If (\ref{cap-will-Hn})'s equality holds, then by analyzing the equality case of \eqref{4.19b}, we see that all on $\partial K_t$ are umbilical. Hence, $K$ is a geodesic ball. Conversely, if $K$ is a geodesic ball, then (\ref{cap-will-Hn})'s equality is automatically true.

\subsubsection*{\underline{Arguments for \eqref{cap-will-Hn2} and \eqref{4.11c}}} This process consists of two parts.

\begin{itemize}
\item The part on \eqref{cap-will-Hn2} will be completed in two steps.
\begin{itemize}
\item The first is to establish the \eqref{Q}'s $[1,\infty)\ni a$-generalization as seen below.
\begin{lem}\label{l4.2}
If $\Sigma_0\subset\mathbb{H}^{n\ge 3}$ is a star-shaped smooth hypersurface with
\begin{equation}
\label{4.18o}
\text{either\
                      $a>1$ \&\ $\Sigma_0$ is h-convex \ or $a=1$ \&\  $\Sigma_0$ is mean-convex, \ }
                    \end{equation}
                        then
\begin{equation}\label{qa}
Q_a(t)=|\Sigma_t|^{\frac{2a}{n-1}-1}\int_{\Sigma_t} ({\rm p}_1^2 - 1)^a \, d\mu_t\  \text{is decreasing along the IMCF $[0,\infty)\ni t\mapsto \Sigma_t$}.
\end{equation}
\end{lem}
\begin{proof}
In fact, upon denoting
$$ f_a({\rm p_1})=({\rm p}_1^2-1)^a \ \ \ \& \ \ \ F_a(t)=\int_{\Sigma_t}  f_a({\rm p}_1) \, d\mu_t,
$$
we can calculate{
$$
\begin{aligned}
\frac{dF_a(t)}{dt}&= \int_{\Sigma_t} f_a'({\rm p_1})\left(\frac{\partial {\rm p_1}}{\partial t}\right)\, d\mu_t +F_a(t)
\\ &=-{(n-1)^{-2}} \int_{\Sigma_t} \Bigg(f_a'({\rm p_1})\Delta({{\rm p_1}}^{-1})+\left(\frac{f_a'({\rm p_1})}{{\rm p_1}}\right)(|\mathsf{h}|^2-(n-1))\Bigg)\, d\mu_t +F_a(t) \\
&=-{(n-1)^{-2}} \int_{\Sigma_t} \left(\frac{f_a''({\rm p_1})}{{\rm p_1}^2}\right)|\nabla {\rm p_1}|^2 \, d\mu_t \\
&\  \  \  \,-{(n-1)^{-2}} \int_{\Sigma_t}\left(\frac{f_a'({\rm p_1})}{{\rm p_1}}\right)\big(|\mathsf{h}|^2-(n-1)\big)\, d\mu_t +F_a(t).
\end{aligned}
$$}
If $a=1,$ then \eqref{4.18o} implies
$$
\begin{cases}
  f_a'({\rm p_1})=2{\rm p}_1\geq0;\\
  f_a''({\rm p_1})=2>0,
  \end{cases}
$$
along the IMCF.
\par By the general tensor maximum principle argument in the proof of Proposition 2.2 in \cite{GWW14}, applied with $F=\sigma_1,$
 horospherical convexity is preserved along the IMCF.
If $a>1,$ then \eqref{4.18o} implies
\begin{equation}\label{f_a'}
  \begin{cases}
  f_a'({\rm p_1})=2a {\rm p_1}({\rm p}_1^2-1)^{a-1}\geq0; \\
   f_a''({\rm p_1})= 2a({\rm p}_1^2-1)^{a-1}+4a(a-1){\rm p}_1^2({\rm p}_1^2-1)^{a-2}\geq0,
\end{cases}
\end{equation}
along the IMCF. Here for the regularity problem when $a\in(1,2),$ one may consider $$
f_a({\rm p_1})=({\rm p}_1^2-1+\varepsilon)^a
$$and let $\varepsilon\rightarrow 0.$
Also, recall the trace inequality $$|\mathsf{h}|^2 \geq (n-1){\rm p}_1^2.$$ So, we can employ \eqref{f_a'} to obtain
\begin{align*}
\frac{dF_a(t)}{dt}&\leq -{(n-1)^{-2}} \int_{\Sigma_t}\left(\frac{f_a'({\rm p_1})}{{\rm p_1}}\right)\big((n-1){\rm p}_1^2-(n-1)\big)\, d\mu_t +F_a(t)\\
&=\left(1-\frac{2a}{n-1}\right)F_a(t),
\end{align*}
whence getting
$$\frac{dQ_a(t)}{dt}\leq 0\  \  \text{thanks to}\  \ |\Sigma_t|=e^t|\Sigma|,
$$
as required within \eqref{qa}.
\end{proof}
\item The second is to reach \eqref{cap-will-Hn2} via Lemma \ref{l4.2}.

First of all, for $a\in [1,\infty)$ along the IMCF, the decreasing property of $Q_a(t)$ in (\ref{qa}) with
$$
\Sigma_0=\partial K\ \ \&\ \ \Sigma_t=\partial K_t,
$$
yields
\begin{equation}\label{4.20}
\begin{cases}
|\partial K_t|^{\frac{2a}{n-1}-1}\int_{\partial K_t} ({\rm p}_1^2 - 1)^a \, d\mu_t \leq | \partial K|^{\frac{2a}{n-1}-1}\int_{ \partial K} ({\rm p}_1^2 - 1)^a \, d\mu;\\
\int_{\partial K_t} ({\rm p}_1^2 - 1)^a \, d\mu_t \leq e^{(1-\frac{2a}{n-1})t}\int_{\partial K} ({\rm p}_1^2 - 1)^a \, d\mu=
 e^t\int_{\partial K} \big(({\rm p}_1^2 - 1)e^{\frac{-2t}{n-1}}\big)^a \, d\mu.
\end{cases}
\end{equation}
Notice that the above estimation \eqref{4.20} also holds for $a=0$. Thus, as an application we get
$$\begin{aligned}
\int_{\partial K_t} {\rm p}_1^{2m} \, d\mu_t &=\int_{\partial K_t} (1+{\rm p}_1^2-1)^m \, d\mu_t
\\ & =\sum_{j=0}^m \binom{m}{j}\int_{\partial K_t} ({\rm p}_1^2-1)^j \, d\mu_t
\\&\leq \sum_{j=0}^m \binom{m}{j}e^t\int_{\partial K} \big(({\rm p}_1^2 - 1)e^{\frac{-2t}{n-1}}\big)^j \, d\mu
\\ &=e^t\int_{\partial K} \big(1+({\rm p}_1^2 - 1)e^{\frac{-2t}{n-1}}\big)^m \, d\mu
\end{aligned}
$$

Next, (\ref{cap-Tp-Hn}), along with the H\"older inequality, implies
$$\begin{aligned}
T_p(t) & \leq \left(\int_{\partial K_t}\big((n-1){\rm p}_1\big)^{2m} \, d\mu_t\right)^{\frac{p-1}{2m}}|\partial K_t|^{1-\frac{p-1}{2m}}
\\ &\leq (n-1)^{p-1}\left(e^t\int_{\partial K} \big(1+({\rm p}_1^2 - 1)e^{\frac{-2t}{n-1}}\big)^m \, d\mu\right)^{\frac{p-1}{2m}}(e^t|\partial K|)^{1-\frac{p-1}{2m}}
\\ &=  (n-1)^{p-1}|\partial K|^{1-\frac{p-1}{2m}}e^t\left(\int_{\partial K} \big(1+({\rm p}_1^2 - 1)e^{\frac{-2t}{n-1}}\big)^m \, d\mu\right)^{\frac{p-1}{2m}}
\end{aligned}
$$
whence
\begin{align}\label{4.21}
{\rm Cap}_{p}(K) &\leq (n-1)^{p-1}|\partial K|^{1-\frac{p-1}{2m}} \left( \int_0^\infty \left(\int_{\partial K} \big(1+({\rm p}_1^2 - 1)e^{\frac{-2t}{n-1}}\big)^m \, d\mu\right)^{\frac{-1}{2m}} e^{\frac{t}{1-p}}\, dt \right)^{1-p}\\
&= \left(\frac{n-1}{p-1}\right)^{p-1}|\partial K|^{1-\frac{p-1}{2m}}  \left( \int_0^1 \left( \int_{\partial K} \left(1+({\rm p}_1^2 - 1) s^{\frac{2(p-1)}{n-1}}\right)^m \, d\mu \right)^{-\frac{1}{2m}} ds\right)^{1-p}\notag
\\
&= \left(\frac{n-1}{p-1}\right)^{p-1}|\partial K|  \left( \int_0^1 \left(\frac{\displaystyle \int_{\partial K} \left(1+({\rm p}_1^2 - 1) s^{\frac{2(p-1)}{n-1}}\right)^m\,d\mu}{|\partial K|} \right)^{-\frac{1}{2m}} ds\right)^{1-p},\notag
\end{align}
whose first equality case can be verified in the same way as that for (\ref{cap-will-Hn}).

Since ${\rm p}_1^2-1\equiv\sinh^{-2}r$ on $\partial\bar B(r)$, the defining identity for the $q$-curvature-excess radius is equivalently
\[
\frac1{\sinh^2\widehat{\rho}_q(K)}
=
\left(
\frac{\displaystyle \int_{\partial K}({\rm p}_1^2-1)^{q/2}\,d\mu}{|\partial K|}
\right)^{\frac2q}.
\]
In particular, for $q=2m$,
\[
\frac1{\sinh^2\widehat{\rho}_{2m}(K)}
=
\left(
\frac{\displaystyle \int_{\partial K}({\rm p}_1^2-1)^m\,d\mu}{|\partial K|}
\right)^{\frac1m}.
\]

Finally, by H\"older's inequality, for every $j\in\{1,\ldots,m\}$,
\[
\left(
\frac{\displaystyle \int_{\partial K}({\rm p}_1^2-1)^j\,d\mu}{|\partial K|}
\right)^{1/j}
\le
\left(
\frac{\displaystyle \int_{\partial K}({\rm p}_1^2-1)^m\,d\mu}{|\partial K|}
\right)^{1/m}
=
\frac{1}{\sinh^2\widehat{\rho}_{2m}(K)}.
\]
Consequently,
\begin{align}\label{4.22}
 \frac{\displaystyle \int_{\partial K} \left(1+({\rm p}_1^2 - 1) s^{\frac{2(p-1)}{n-1}}\right)^m\,d\mu}{|\partial K|}
&=
\sum_{j=0}^m \binom{m}{j}
\left(
\frac{\displaystyle \int_{\partial K}({\rm p}_1^2-1)^j\,d\mu}{|\partial K|}
\right)
s^{\frac{2(p-1)j}{n-1}}
\\
&\le
\sum_{j=0}^m\binom{m}{j}
\frac{s^{\frac{2(p-1)j}{n-1}}}{\sinh^{2j}\widehat{\rho}_{2m}(K)}
\notag\\
&=
\left(
1+\frac{s^{\frac{2(p-1)}{n-1}}}{\sinh^2\widehat{\rho}_{2m}(K)}
\right)^m,\notag
\end{align}
whose inequality becomes equality iff $K$ is a geodesic ball or $m=1$.

Combining \eqref{4.21} and \eqref{4.22} gives \eqref{cap-will-Hn2}. To identify it with the geodesic-ball ratio, make the change of variables
\[
\sinh r=\sinh\widehat{\rho}_{2m}(K)\,s^{-\frac{p-1}{n-1}}\ \  \forall\ \ s\in (0,1].
\]
Exactly as above, one has
\[
\left(\frac{n-1}{p-1}\right)^{p-1}
\left(
\int_0^1
\left(
1+\frac{s^{\frac{2(p-1)}{n-1}}}{\sinh^2\widehat{\rho}_{2m}(K)}
\right)^{-\frac12}ds
\right)^{1-p}
=
\frac{{\rm Cap}_{p}\big(\bar B(\widehat{\rho}_{2m}(K))\big)}
{|\partial\bar B(\widehat{\rho}_{2m}(K))|}.
\]

 In particular, for $p=2=n-1$ there holds
\begin{equation}\label{4.23}
{\rm Cap}_{2}(K) \leq | \partial K| \left(1+\sqrt{\frac{\int_{ \partial K}{\rm p}_1^2\, d\mu }{|\partial K|}}\right).
\end{equation}
\end{itemize}

\item Finally, \eqref{4.11c} follows from \eqref{cap-will-Hn} by interpolation. Put $q=p-1>2m$ and consider the probability space $(\partial K,d\mu/|\partial K|)$ with
\[
X={\rm p}_1^2-1\ge0.
\]
Since $$X\equiv\sinh^{-2}r\ \ \text{on}\ \ \partial\bar B(r),
$$
the defining identity for $\rho_{p,m}(K)$ is equivalently
\begin{align*}
\frac1{\sinh^2\rho_{p,m}(K)}
&=
\big(\|{\rm p}_1\|_\infty^2-1\big)^{1-\frac{2m}{p-1}}
\left(
\frac{\displaystyle \int_{\partial K}X^m\,d\mu}{|\partial K|}
\right)^{\frac{2}{p-1}}\\
&=
\big(\|{\rm p}_1\|_\infty^2-1\big)^{1-\frac{2m}{p-1}}
\left(\frac1{\sinh^2\widehat{\rho}_{2m}(K)}\right)^{\frac{2m}{p-1}}.
\end{align*}
By Minkowski's inequality,
\begin{align*}
\frac1{\sinh^2\rho_q(K)}
&=
\left(
\frac{\displaystyle \int_{\partial K}(1+X)^{q/2}\,d\mu}{|\partial K|}
\right)^{2/q}-1
\le
\left(
\frac{\displaystyle \int_{\partial K}X^{q/2}\,d\mu}{|\partial K|}
\right)^{2/q}.
\end{align*}
Since $q/2>m$, the $L^m$--$L^\infty$ interpolation inequality gives
\begin{align*}
\left(
\frac{\displaystyle \int_{\partial K}X^{q/2}\,d\mu}{|\partial K|}
\right)^{2/q}
&\le
\|X\|_\infty^{1-\frac{2m}{q}}
\left(
\frac{\displaystyle \int_{\partial K}X^m\,d\mu}{|\partial K|}
\right)^{2/q}
\\
&=
\big(\|{\rm p}_1\|_\infty^2-1\big)^{1-\frac{2m}{p-1}}
\left(\frac1{\sinh^2\widehat{\rho}_{2m}(K)}\right)^{\frac{2m}{p-1}}
\\
&=
\frac1{\sinh^2\rho_{p,m}(K)}.
\end{align*}
Consequently,
\[
\rho_{p-1}(K)\ge \rho_{p,m}(K).
\]
Since
\[
r\longmapsto\frac{{\rm Cap}_p(\bar B(r))}{|\partial\bar B(r)|}
\]
is decreasing, \eqref{cap-will-Hn} immediately implies \eqref{4.11c}. Equality in \eqref{4.11c} forces equality in \eqref{cap-will-Hn}, and hence $K$ is a geodesic ball; the converse is immediate.
\end{itemize}

\end{proof}

\begin{rem} Six comments are in order.

\begin{itemize}

\item \eqref{cap2-H3}'s first part is the hyperbolic counterpart of Szeg\"{o}'s classical result in \cite{Sze31}. Also, note that the relative volume in $\mathbb{H}^3$ satisfies
$$
{\rm RV}(K)=1+(2\pi)^{-1}\int_{\partial K}(1+{\rm p}_1)\,d\mu.
$$
Thus, \eqref{cap2-H3}'s first part provides a refinement of \eqref{cap-rV} under $n-3=0=p-2$. Meanwhile, the endpoint $p\to1$ of \eqref{cap2-H3}'s second inequality reads
$$
{\rm RCap}_1(K)\le(n-1)^{-1}\ln\left(\omega_{n-1}^{-1}\int_{\partial K}(1+{\rm p}_1)^{n-1}\,d\mu\right),
$$
with equality if $K$ is a geodesic ball.

\item The Cauchy--Schwarz inequality implies
$$
\sqrt{\frac{\int_{\partial K}{\rm p}_1^2\,d\mu}{|\partial K|}}
\ge
\frac{\int_{\partial K}{\rm p}_1\,d\mu}{|\partial K|}.
$$
Therefore, \eqref{4.23} is weaker than \eqref{cap2-H3}'s first inequality; however, the latter holds under the stronger assumption of horospherical convexity. In dimension $n=2$, the same curvature-radius form of \eqref{cap-will-Hn} is consistent with \cite[Theorem 4]{LLX}; namely, for $p\ge3$,
$$
\frac{{\rm Cap}_p(K)}{|\partial K|}
\le
\frac{{\rm Cap}_p\big(\bar B(\rho_{p-1}(K))\big)}
{|\partial\bar B(\rho_{p-1}(K))|},
$$
with equality iff $K$ is a geodesic ball. This may be viewed as an uncertainty principle for the triple
$$
\left\{{\rm Cap}_p,\ W_1,\ \rho_{p-1}\right\}.
$$

\item The two effective radii agree at the quadratic level:
$$
\rho_2(K)=\widehat{\rho}_2(K),
$$
because
$$
\frac1{\sinh^2\rho_2(K)}
=
\frac{\displaystyle \int_{\partial K}({\rm p}_1^2-1)\,d\mu}{|\partial K|}
=
\frac1{\sinh^2\widehat{\rho}_2(K)}.
$$
Consequently, \eqref{cap-will-Hn2}'s case $m=1$ coincides with \eqref{cap-will-Hn}'s case $p=3$.

\item If $3\le p\le2m+1$, then Minkowski's inequality followed by H\"older's inequality yields
\begin{align*}
\frac1{\sinh^2\rho_{p-1}(K)}
&=
\left(
\frac{\displaystyle \int_{\partial K}{\rm p}_1^{p-1}\,d\mu}{|\partial K|}
\right)^{\frac{2}{p-1}}-1
\\
&\le
\left(
\frac{\displaystyle \int_{\partial K}({\rm p}_1^2-1)^{\frac{p-1}{2}}\,d\mu}{|\partial K|}
\right)^{\frac{2}{p-1}}
\\
&\le
\left(
\frac{\displaystyle \int_{\partial K}({\rm p}_1^2-1)^m\,d\mu}{|\partial K|}
\right)^{\frac1m}\\
&=
\frac1{\sinh^2\widehat{\rho}_{2m}(K)}.
\end{align*}
Hence
$$
\rho_{p-1}(K)\ge\widehat{\rho}_{2m}(K).
$$
Moreover, the function
$$
r\longmapsto
\frac{{\rm Cap}_p(\bar B(r))}{|\partial\bar B(r)|}
$$
is decreasing, as is clear from the integral representation used in the proof of \eqref{cap-will-Hn}. It follows that \eqref{cap-will-Hn} is always sharper than the geodesic-ball bound in \eqref{cap-will-Hn2}. The stronger full-moment estimate \eqref{4.21}, however, retains all moments up to order $m$ and may therefore yield a sharper estimate.

\item Consider the probability space $(\partial K,d\mu/|\partial K|)$ and the random variable
$$
X={\rm p}_1^2-1\ge0.
$$
Then the full-moment estimate \eqref{4.21} can be written as
\begin{align*}
\frac{{\rm Cap}_{p}(K)}{|\partial K|}
&\le
\left(
\Big(\frac{p-1}{n-1}\Big)\int_0^1
\left(
\mathbb E\!\left[\left(1+Xs^{\frac{2(p-1)}{n-1}}\right)^m\right]
\right)^{-\frac1{2m}}ds
\right)^{1-p}
\\
&\le
\frac{{\rm Cap}_p\big(\bar B(\widehat{\rho}_{2m}(K))\big)}
{|\partial\bar B(\widehat{\rho}_{2m}(K))|},
\end{align*}
where
$$
\frac1{\sinh^2\widehat{\rho}_{2m}(K)}
=
\big(\mathbb E X^m\big)^{1/m}.
$$
Thus $\widehat{\rho}_{2m}(K)$ is precisely the geodesic radius determined by the $m$-th moment of the squared mean-curvature excess.

\item For $p>2m+1$, the radius $\rho_{p,m}(K)$ in \eqref{4.11c} interpolates between the finite-moment radius $\widehat{\rho}_{2m}(K)$ and the supremum curvature scale. Indeed,
\[
\begin{cases}
\lim_{p\downarrow 2m+1}\rho_{p,m}(K)=\widehat{\rho}_{2m}(K);\\
\lim_{p\to\infty}\coth\rho_{p,m}(K)=\|{\rm p}_1\|_\infty.
\end{cases}
\]
Moreover, the proof of \eqref{4.11c} gives
\[
\rho_{p-1}(K)\ge\rho_{p,m}(K),
\]
so \eqref{cap-will-Hn} is sharper whenever the full $(p-1)$-st curvature moment is retained, whereas \eqref{4.11c} uses only the $2m$-th curvature-excess radius together with the $L^\infty$ curvature scale. Thus \eqref{4.11c} provides a direct geodesic-ball interpolation between the finite-moment and supremum regimes.

\end{itemize}

\end{rem}

\end{document}